\documentclass{article}
\usepackage[margin=2.5cm]{geometry}

\usepackage{graphicx}
\usepackage{amsmath,amssymb,commath,mathtools,mathrsfs}
\usepackage{amsthm,authblk}
\usepackage{wasysym}
\usepackage{color}
\usepackage{bm}
\usepackage{algorithm, algpseudocode, algorithmicx} %
\algrenewcommand{\algorithmiccomment}[2][.25\linewidth]{%
  \leavevmode\hfill\makebox[#1][l]{$\triangleright$~#2}} 
\usepackage{caption}
\usepackage{subcaption}
\usepackage{paralist}

\usepackage{booktabs,multirow}

\usepackage{natbib}
 \bibpunct[, ]{[}{]}{,}{n}{}{,}%

\usepackage{xcolor}

\usepackage[colorlinks=true,breaklinks=true,bookmarks=true,urlcolor=blue,
     citecolor=blue,linkcolor=blue,bookmarksopen=false,draft=false]{hyperref}

\DeclareMathOperator*{\argmax}{arg\,max}
\DeclareMathOperator*{\argmin}{arg\,min}

\newcommand{\nVert}[1]{\lVert#1\rVert}
\newcommand{\innerprod}[2]{\left\langle#1,#2\right\rangle}
\newcommand{\bangle}[2]{\langle#1,#2\rangle}
\newcommand{\transpose}{\mathsf{T}}

\newcommand{\conv}{\mathrm{conv}}

\newcommand{\Reg}{\mathrm{R}}

\newcommand{\Vol}{\mathrm{Vol}}

\makeatletter
\def\ubar#1{\underline{\sbox\tw@{$#1$}\dp\tw@\z@\box\tw@}}
\makeatother

\makeatletter
\def\defcal#1{%
\expandafter\newcommand\csname cal#1\endcsname{\mathcal{#1}}}
\count@=0
\loop
\advance\count@ 1
\edef\y{\@Alph\count@}%
\expandafter\defcal\y
\ifnum\count@<26
\repeat
\def\defscr#1{%
\expandafter\newcommand\csname scr#1\endcsname{\mathscr{#1}}}
\count@=0
\loop
\advance\count@ 1
\edef\y{\@Alph\count@}%
\expandafter\defscr\y
\ifnum\count@<26
\repeat
\def\defbb#1{%
\expandafter\newcommand\csname bb#1\endcsname{\mathbb{#1}}}
\count@=0
\loop
\advance\count@ 1
\edef\y{\@Alph\count@}%
\expandafter\defbb\y
\ifnum\count@<26
\repeat
\makeatother

\renewcommand{\phi}{\varphi}
\renewcommand{\epsilon}{\varepsilon}

\newcommand{\ulF}{\underline{F}}
\newcommand{\olF}{\overline{F}}
\newcommand{\ulcQ}{\underline{\calQ}}
\newcommand{\olcQ}{\overline{\calQ}}

\newcommand{\UB}{\text{\textsc{UpperBound}}}
\newcommand{\LB}{\text{\textsc{LowerBound}}}
\newcommand{\NumEval}{\textup{\texttt{\#Eval}}}
\newcommand{\CDDP}{\textup{CDDP}}
\newcommand{\NDDP}{\textup{NDDP}}

\newtheorem{theorem}{Theorem}
\newtheorem{lemma}{Lemma}
\newtheorem{proposition}{Proposition}
\newtheorem{corollary}{Corollary}

\newtheorem{remark}{Remark}
\newtheorem{question}{Question}
\newtheorem{example}{Example}
\newtheorem{definition}{Definition}

\begin{document}

\title{On Distributionally Robust Multistage Convex Optimization: New Algorithms and Complexity Analysis}
\author[1]{Shixuan Zhang}
\author[2]{Xu Andy Sun}
\affil[1]{Wm Michael Barnes ‘64 Department of Industrial \& Systems Engineering, Texas A\&M University, College Station, TX 77843, USA \quad shixuan.zhang@tamu.edu}           
\affil[2]{Operations Research and Statistics, Sloan School of Management, Massachusetts Institute of Technology, Cambridge, MA 02139, \quad sunx@mit.edu}
\renewcommand\Affilfont{\itshape\small}
\date{}
\maketitle

\begin{abstract}
    This paper presents an algorithmic study and complexity analysis for solving distributionally robust multistage convex optimization (DR-MCO) problems. 
    Our main contribution is a novel \emph{nonconsecutive} dual dynamic programming (NDDP) algorithm which explores different stages in an adaptive fashion.
    In contrast with the usual consecutive dual dynamic programming (CDDP) algorithm, we show that NDDP reduces the subproblem complexity from quadratic to linear dependency on the number of stages.
    Two different DR-MCO examples are also presented to show the efficiency and effectiveness of the proposed NDDP algorithm.
\end{abstract}

\maketitle

\section{Introduction}
\label{sec:Introduction}
Distributionally robust multistage convex optimization (DR-MCO) is a sequential decision making problem with convex objective functions and constraints, where the exact probability distribution of the uncertain parameters is unknown and decisions need to be made considering a family of distributions. 
DR-MCO provides a unified framework for studying decision-making under uncertainty. 
It includes as special cases both multistage stochastic convex optimization (MSCO), where some distribution of the uncertain parameters is known and the expected cost is to be minimized, and multistage robust convex optimization (MRCO), where the uncertainty is described by a set and the worst-case cost is to be minimized.  
DR-MCO as a general decision framework finds ubiquitous applications in energy system planning, supply chain and inventory control, and other areas (see~e.g. \cite{shapiro2009lectures,bental2009robust}).

Both MSCO and MRCO are in general challenging to solve, due to the fast growth of the number of decisions with respect to the number of decision stages \cite{shapiro_complexity_2006,dyer_computational_2006,georghiou_robust_2019}. 
Meanwhile, real-world problems are often endowed with special structures in the uncertainty. In particular, the uncertainty may exhibit stagewise independence (SI), i.e., the uncertainty in different stages are independent from each other. Many uncertainty structures, such as autoregressive stochastic models, can be reformulated to satisfy SI~\cite{shapiro2012final}. This versatile modeling capability of SI has great implications on computation. It allows recursive formulations of cost-to-go functions in each stage to be independent of the outcomes in its previous stages, thus making efficient approximations of the cost-to-go functions possible. Indeed, SI has been successfully exploited by various algorithms in solving MSCO and MRCO~\cite{shapiro_analysis_2011,philpott_convergence_2008,georghiou_robust_2019,philpott_midas_2016,baucke_deterministic_nodate,baucke_deterministic_nodate-1,ahmed2019stochastic,zou_stochastic_2019,shapiro_stationary_nodate}.

Dual dynamic programming (DDP) is a prevalent class of recursive cutting plane algorithms for MSCO and MRCO that effectively exploits the SI structure.
The earliest form of DDP originates from nested Benders decomposition with stochastic sampling steps~\cite{birge_decomposition_1985,pereira_stochastic_1985,pereira_multi-stage_1991}.
Thus it was known as stochastic dual dynamic programming (SDDP) and has been widely adopted in energy system planning~\cite{flach_long-term_2010,de_matos_improving_2015,shapiro_risk_2013}.
The deterministic version of DDP was later proposed, which uses both over- and under-approximations for sampling and termination~\cite{baucke_deterministic_nodate,baucke_deterministic_nodate-1}.
Meanwhile, robust dual dynamic programming (RDDP) is developed for multistage robust linear optimization~\cite{georghiou_robust_2019}.
Similar to the deterministic DDP, RDDP constructs deterministic approximations to select the worst-case outcome, and thus has guaranteed optimality, in contrast to decision rule-based methods~\cite{kuhn2011primal,bertsimas2015design,hadjiyiannis2011scenario}. 
Recently, DDP has been further extended to some DR-MCO models with statistical or heuristic stopping criteria~\cite{anderson_improving_2019,huang_study_2017,philpott_distributionally_2018,duque2019distributionally}.

The convergence analysis of DDP begins with multistage linear optimization~\cite{philpott_convergence_2008,shapiro_analysis_2011,linowsky_convergence_2005,chen_convergent_1999}, where an almost sure finite convergence is established based on polyhedral structures.
In~\cite{girardeau_convergence_2015}, an asymptotic convergence is proved for MSCO problems.
Due to the multistage structure, a main complexity question concerning DDP is the dependency of its iteration complexity on the number of decision stages, which is recently answered for MSCO, independently in~\cite{lan2022complexity,zhang2022stochastic}. 
However, it is not yet known whether the DDP algorithm, together with the complexity analysis, works for DR-MCO. 

One major issue of extending the algorithm and the analysis from MSCO to DR-MCO is the lack of proper termination criterion.
Due to the distributional uncertainty in the model, the commonly used statistical upper bound for the policy evaluation in the MSCO literature (e.g., \cite{shapiro_analysis_2011,zou_stochastic_2019}) is no longer valid for the DR-MCO problems.
As a result, computational experiments in~\cite{philpott_distributionally_2018} and \cite{duque2019distributionally} terminate at a fixed number of iterations or cuts.
To overcome the lack of statistical upper bound, we follow the idea of deterministic upper bounds~\cite{philpott_solving_2013,baucke_deterministic_nodate} and focus on dynamically constructing both under- and over-approximations to certify the optimality of the solutions, similar to the RDDP algorithm~\cite{georghiou_robust_2019}.
Such deterministic termination criteria and approximation updates demand the study of \emph{subproblem complexity} of the DDP algorithm.
\begin{question}
How many single-stage subproblems do we need to solve for approximation updates before reaching a satisfying optimality gap for the DR-MCO?
\end{question}

For MSCO, existing analysis~\cite{lan2022complexity,zhang2022stochastic} predicts an iteration complexity bound of $\calO(T)$, where in each iteration at least $T$ single-stage subproblems need to be solved, leading to an $\calO(T^2)$ subproblem complexity bound for MSCO.
Our paper answers this complexity question, not only by extending the above MSCO analysis to DR-MCO, but more importantly, by proposing a new \emph{nonconsecutive} DDP (NDDP) that reduces $\calO(T^2)$ to the optimal $\calO(T)$ dependency on the number of stages $T$.
This complexity result resolves an open conjecture in~\cite{georghiou_robust_2019}, on the efficiency of nonconsecutive implementation on MRCO problem, in the more general DR-MCO setting.
Finally, upon the completion of this paper, a work of similar spirit~\cite{ju2023dual} studies the complexity of finite-horizon approximations of a stationary infinite-horizon MSCO problem, where they also successfully show a linear dependency on the approximation horizon length.
We remark that our analysis is original and remains the only subproblem complexity result that works for non-stationary finite-horizon DR-MCO problems to date.

The rest of the paper is organized as follows.
Section~\ref{sec:Formulations} contains the formulations of the problems and the discussion on a technical challenge of bounding the dual variables in DDP recursions.
In Section~\ref{sec:Algorithms}, we define single stage subproblem oracles, review the natural extension of CDDP to DR-MCO, and then present the new NDDP algorithm with improved subproblem oracle complexity bounds.
In Section~\ref{sec:Numerical}, we present two classes of numerical examples that demonstrate the effectiveness of the algorithms. 
Concluding remarks are made in Section~\ref{sec:Conclusion}.

\section{Formulations and Recursive Approximation}
\label{sec:Formulations}
In this section, we introduce the formulation and some basic properties of our distributionally robust multistage convex optimization (DR-MCO) model.
We then discuss the dual-bounding technique and its exactness needed for our complexity analysis in Section~\ref{sec:Algorithms}.

\subsection{Problem Formulations}
\label{subsec:ProblemFormulations}
We present our definition of DR-MCO under the assumption of stagewise independence (SI).
The SI assumption is necessary for the algorithmic development in this paper:
while more general settings can be found in~\cite{shapiro2009lectures}, the SI assumption allows us to avoid the exponential growth of problem sizes with respect to the number of stages, as analyzed in~\cite{shapiro_complexity_2006,reaiche2016note} for MSCO problems.
To illustrate the practicality of our DR-MCO model, we show in Section~\ref{sec:FiniteUncertaintySet} that most MSCO and MRCO in the existing literature (such as~\cite{girardeau_convergence_2015} and ~\cite{georghiou_robust_2019}) can be encompassed within this DR-MCO framework even with only finitely many uncertainty outcomes in each stage.

Let \(\calT\coloneqq\{1,2,\dots,T\}\) denote the set of stage indices and \(\calT'\coloneqq\calT\setminus\{1\}\). 
For each \(t\in\calT\), the decision variable in stage \(t\) is denoted as \(x_t\) and constrained in a compact convex set \(\calX_t\subset\bbR^{d_t}\) with dimension \(d_t\in\bbZ_{\ge0}\).
The uncertainty in stage \(t\) is modeled as a random vector \(\xi_t\) with some possibly unknown probability distribution and its convex support set denoted as \(\Xi_t\).
For simplicity, we use \(x_0\) and \(\xi_1\) to denote deterministic parameters, known as the initial conditions.
The cost incurred by the decisions and the uncertainty outcomes in stage \(t\) is modeled by a nonnegative lower semicontinuous function \(f_t(x_{t-1},x_t;\xi_t)\), that is convex in \((x_{t-1},x_t)\) for each \(\xi_t\in\Xi_t\) and allowed to take \(+\infty\) for infeasibility.
We show by the following example that the usual constraints can be modeled as part of the cost functions \(f_t\).
\begin{example}\label{ex:ConstrainedOptimization}
Let \(\calY_t\) denote a compact convex decision set of auxiliary variables \(y_t\) in stage \(t\in\calT\).
Given a nonnegative, convex, continuous real-valued cost function \(f_t^{<\infty}(x_{t-1},y_t,x_t;\xi_t)\) and any convex continuous functional feasibility constraints \(g_t(x_{t-1},y_t,x_t;\xi_t)\le 0\), we can define \(\calF_t:=\{(x_{t-1},x_t,\xi_t)\in\calX_{t-1}\times\calX_t\times\Xi_t:\exists\,y_t\in\calY_t\text{ s.t. }g_t(x_{t-1},y_t,x_t,\xi_t)\le 0\}\) and set the extended real-valued cost function as 
\[
    f_t(x_{t-1},x_t;\xi_t)=\begin{cases}
        \min_{y_t\in\calY_t}f_t^{<\infty}(x_{t-1},y_t,x_t;\xi_t),&\text{ if }(x_{t-1},x_t,\xi_t)\in\calF_t,\\
        +\infty,&\text{ otherwise}.
    \end{cases}
\]
It can be checked that the cost function \(f_t\) in this example is indeed nonnegative, lower semicontinuous, convex in \((x_{t-1},x_t)\), and proper if \(\calF_t\neq\varnothing\).
\end{example}

Let \(\calP_t\) denote a set of Borel probability measures supported on the uncertainty set \(\Xi_t\), called the ambiguity set, for \(t\in\calT'\).
Our DR-MCO is defined as
\begin{align}\label{eq:DR-MSCP-Def}
    \min_{x_1\in\calX_1}\quad f_1(x_0,x_1;\xi_1)\ +&\sup_{p_2\in\calP_2}\bbE_{\xi_2\sim p_2} \min_{x_2\in\calX_2}\Bigg[f_2(x_1,x_2;\xi_2) + \\
    &+\sup_{p_3\in\calP_3}\bbE_{\xi_3\sim p_3} \min_{x_3\in\calX_3}\bigg[f_3(x_2,x_3;\xi_3)+\cdots\notag\\
    &+\sup_{p_T\in\calP_{T}}\bbE_{\xi_T\sim p_T} \min_{x_T\in\calX_T}f_T(x_{T-1},x_T;\xi_T)\bigg]\Bigg].\notag
\end{align}
Here, each expectation \(\bbE_{\xi_t\sim p_t}\) is taken with respect to some given probability measure \(p_t\in\calP_t\).
Our problem~\eqref{eq:DR-MSCP-Def} relies on the SI assumption in the sense that the ambiguity sets and the expectations are independent across stages.
Consequently, we can define the (worst-case expected) cost-to-go functions recursively from \(t=T\) to \(t=2\) as:
\begin{equation}\label{eq:DR-MSCP-Recursion}
    \calQ_{t-1}(x_{t-1})\coloneqq\sup_{p_t\in\calP_t}\bbE_{\xi_t\sim p_t}\min_{x_t\in\calX_t}f_t(x_{t-1},x_{t};\xi_{t})+\calQ_{t}(x_{t}),
\end{equation}
with \(\calQ_T(x_T)\equiv0\) for any \(x_T\in\calX_T\).
To simplify the notation, we also define 
\begin{equation}\label{eq:DR-MSCP-ValueFunction}
    Q_t(x_{t-1};\xi_t)\coloneqq\min_{x_t\in\calX_t}f_t(x_{t-1},x_t;\xi_t)+\calQ_t(x_t)
\end{equation}
as the value function,
such that \(\calQ_{t-1}(x_{t-1})=\sup_{p_t\in\calP_t}\bbE_{\xi_t\sim p_t} Q_t(x_{t-1};\xi_t)\) for all \(t\in\calT'\).
Our algorithmic goal in this paper is to find a near-optimal first stage solution
\begin{equation}
    x_1^*\in\argmin_{x_1\in\calX_1} f_1(x_0,x_1;\xi_1)+\calQ_1(x_1),
\end{equation}
as a multistage solution may already take exponentially many operations in \(T\) to be written down, e.g., \(2^{T-1}\) if we have an MSCO with \(\abs{\Xi_t}=2\) for \(t=2,\dots,T\).
We say that \(x_1^*\) is an \(\epsilon\)-optimal first stage solution if 
\begin{equation}\label{eq:DR-MSCP-NearOptimalSolution}
    f_1(x_0,x_1^*;\xi_1)+Q_1(x_1^*)\le \varepsilon+\Big(\min_{x_1\in\calX_1}f_1(x_0,x_1;\xi_1)+\calQ_1(x_1)\Big).
\end{equation}

It is well-known that if the ambiguity set \(\calP_t\) is a singleton set, then the supremum is superficial so the DR-MCO~\eqref{eq:DR-MSCP-Def} reduces to an MSCO.
If the ambiguity set \(\calP_t\) contains all Dirac atomic measures (i.e., measures \(\delta_{\hat{\xi}}\) for all \(\hat{\xi}\in\Xi_t\) such that \(\bbE_{\xi_t\sim\delta_{\hat{\xi}}}g(\xi)=g(\hat{\xi})\) for any Borel measurable function \(g\)), then the supremum 
\[
    \sup_{p_t\in\calP_t}\bbE_{\xi_t\sim p_t} Q_t(x_{t-1};\xi_t)=\sup_{\xi_t\in\Xi_t} Q_t(x_{t-1};\xi_t)
\]
and thus the DR-MCO~\eqref{eq:DR-MSCP-Def} reduces to an MRCO.
The following proposition checks that the minimization problems in DR-MCO~\eqref{eq:DR-MSCP-Def} are convex and attained.

\begin{proposition}\label{prop:RecursionConvexity}
    In each stage \(t\in\calT\), the value function \(Q_t(x_{t-1};\xi_t)\) and the cost-to-go function \(\calQ_t(x_t)\) are both lower semicontinuous (lsc) and convex.
\end{proposition}
\begin{proof}
    We prove by recursion from \(t=T\) to \(t=1\).
    By definition, \(\calQ_T(x_T)\equiv0\) is lsc and convex.
    Now assume \(\calQ_t\) is lsc and convex for some \(t\in\calT\).
    Then the sum \(f_t(x_{t-1},x_t;\xi_t)+\calQ_t(x_t)\) is also lsc and convex in \((x_{t-1},x_t)\) for any \(\xi_t\in\Xi_t\).
    Since \(\calX_t\) is compact, we have \(Q_t(x_{t-1};\xi_t):=\min_{x_t\in\calX_t} f_t(x_{t-1},x_t;\xi_t)+\calQ_t(x_t)\) is lsc (see e.g., Lemma 1.30 in~\cite{bauschke2011convex}) and convex.
    Now fix any Borel probability measure \(p_t\in\calP_t\) and take any sequence \(\{x^i\}\subset\calX_{t-1}\) with \(\lim_{i\to\infty}x^i=x_{t-1}\).
    Note that \(Q_t\) is nonnegative by definition, so by Fatou's lemma (see e.g., Lemma 1.28 in~\cite{rudin1987real}) we have
    \[
        \liminf_{i\to\infty}\bbE_{\xi_t\sim p_t} Q_t(x^i;\xi_t)
        \ge\int_{\Xi_t}\liminf_{i\to\infty}Q_t(x^i;\xi_t)\dif p_t(\xi_t)
        \ge\int_{\Xi_t}Q_t(x_{t-1};\xi_t)\dif p_t(\xi_t).
    \]
    The expectation and the integrals are well-defined since \(Q_t\) is lsc, hence Borel measurable.
    This inequality shows that the function \(\bbE_{\xi_t\sim p_t}Q_t(x_{t-1};\xi_t)\) is lsc in \(x_{t-1}\).
    It is also convex in \(x_{t-1}\) by the linearity and monotonicity of expectations (see e.g., Theorem 7.46 in~\cite{shapiro2009lectures}).
    Finally, the epigraph of \(\calQ_{t-1}(\cdot)\) is the intersection of epigraphs of \(\bbE_{\xi_t\sim p_t}Q_t(\cdot;\xi_t)\) for all \(p_t\in\calP_t\), which shows that \(\calQ_{t-1}(\cdot)\) is lsc and convex.
\end{proof}
The lower semicontinuity of the value functions \(Q_t\) implies their Borel measurability, which ensures that the expectations in~\eqref{eq:DR-MSCP-Def} are well-defined.

\subsubsection{Finite Uncertainty Sets}
\label{sec:FiniteUncertaintySet}

If the uncertainty set \(\Xi_t=\{\hat{\xi}_{t,1},\dots,\hat{\xi}_{t,n_t}\}\) is finite, then the ambiguity set is a subset of a \(n_t\)-dimensional simplex \(\calP_t\subseteq\Delta^{n_t}:=\{p_t\in\bbR^{n_t}_{\ge0}:\sum_{k=1}^{n_t}p_{t,k}=1\}\).
In this case, the expectations in the definitions~\eqref{eq:DR-MSCP-Def} and~\eqref{eq:DR-MSCP-Recursion} become finite weighted summations.
For MSCO, such $\Xi_t$ often arise from sample average approximations~\cite{shapiro_analysis_2011}.
It turns out that for MRCO, we can also restrict our attention to a finite support set when the problem is polyhedral (cf.~\cite{georghiou_robust_2019}).

\begin{proposition}\label{prop:ReformulationRobustProgram}
    Let \(\calQ_t\) denote the cost-to-go functions of an MRCO, i.e., a DR-MCO~\eqref{eq:DR-MSCP-Def} with all Dirac atomic measures included in \(\calP_t\) for all \(t\in\calT'\).
    If the cost functions \(f_t\) are jointly convex in \((x_t,\xi_t)\) for each \(x_{t-1}\in\calX_{t-1}\) and the uncertainty sets \(\Xi_t\) are polytopes for all \(t\in\calT'\), then 
    \[
        \calQ_{t-1}(x_{t-1})=\max_{p_t\in\Delta^{n_t}}\sum_{k=1}^{n_t}p_{t,k}\cdot Q_t(x_{t-1};\hat{\xi}_{t,k}),
    \]
    where \(\{\hat{\xi}_{t,1},\dots,\hat{\xi}_{t,n_t}\}\) is the finite set of extreme points of \(\Xi_{t}\).
\end{proposition}
\begin{proof}
    By Proposition~\ref{prop:RecursionConvexity}, each worst-case cost-to-go function \(\calQ_t\) is convex for any \(t\in\calT\).
    From the assumption, the value function \(Q_t(x_{t-1};\xi_t)\) is thus also convex in \(\xi_t\) for each \(x_{t-1}\in\calX_{t-1}\).
    Consequently, we have
    \[
        \calQ_{t-1}(x_{t-1})=\sup_{\xi_t\in\Xi_t}Q_t(x_{t-1};\xi_t)=\max_{p_t\in\Delta^{n_t}}\sum_{k=1}^{n_t}p_{t,k}Q_t(x_{t-1};\hat{\xi}_{t,k}),
    \]
    due to the linearity of the maximization.
    We point out that the above equation holds even when \(\sup_{\xi_t\in\Xi_t}Q_t(x_{t-1};\xi_t)=+\infty\) because in this case we must have \(Q_t(x_{t-1};\hat{\xi}_{t,k})=+\infty\) for some \(k=1,\dots,n_t\) by its convexity.
\end{proof}

In view of Proposition~\ref{prop:ReformulationRobustProgram}, we decide to focus on finite uncertainty sets for numerical implementations in this paper (in Sections~\ref{subsec:SSSOImplementation} and~\ref{sec:Numerical}) for simplicity.
We remark that the proposed algorithms and the complexity bounds are intended for more general DR-MCO problems possibly with infinite uncertainty sets.
These include, for instance, DR-MCO problems with Wasserstein ambiguity sets, which are discussed in our upcoming paper~\cite{zhang2022distributionally}.

\subsection{Approximation of Recursions}
\label{subsec:RecursionApproximation}

We now discuss the approximation of functions \(\calQ_{t-1}\) and \(Q_t\) in the recursions~\eqref{eq:DR-MSCP-Recursion} and \eqref{eq:DR-MSCP-ValueFunction}.
The following lemma relates the Lipschitz continuity of the value functions \(Q_t(\cdot;\xi_t)\) for \(\xi_t\in\Xi_t\) and that of the cost-to-go function \(\calQ_{t}\).
\begin{lemma}\ \label{lemma:RecursionLipschitzContinuity}
    If \(Q_t(\cdot;\xi_t)\) is \(L_t\)-Lipschitz continuous on \(\calX_t\) for any \(\xi_t\in\Xi_t\) and for some \(L_t>0\), then so is \(\calQ_{t}(\cdot)\). 
\end{lemma}
\begin{proof}
    Take any two points \(x^1,x^2\in\calX_{t}\).
    For any \(\epsilon>0\), take \(p_t^1\in\calP_t\) such that
    \[
        \calQ_{t-1}(x^1)-\epsilon \le \bbE_{\xi_t\sim p_t^{1}} Q_t(x^1;\xi_t)\le \calQ_{t-1}(x^1).
    \]
    Then we have
    \begin{align*}
        \calQ_{t-1}(x^1)-\calQ_{t-1}(x^2)
        &\le\epsilon+\bbE_{\xi_t\sim p_t^1}Q_t(x^1;\xi_t)-\sup_{p_t\in\calP_t}\bbE_{\xi_t\sim p_t}Q_t(x^2;\xi_t)\\
        &\le\epsilon+\bbE_{\xi_t\sim p_t^1}\Bigl[Q_t(x^1;\xi_t)-Q_t(x^2;\xi_t)\Bigr]\\
        &\le\epsilon+\bbE_{\xi_t\sim p_t^1}L_t\nVert{x^1-x^2}
        =\epsilon+L_t\nVert{x^1-x^2}.
    \end{align*}
    Since \(\epsilon>0\) is arbitrary, we have \(\calQ_{t-1}(x^1)-\calQ_{t-1}(x^2)\le L_t\nVert{x^1-x^2}\).
    The proof is then completed by exchanging \(x^1\), \(x^2\) and repeating the argument.
\end{proof}

Combining Lemma~\ref{lemma:RecursionLipschitzContinuity} and Proposition~\ref{prop:RecursionConvexity}, we know that  if the value functions are convex and Lipschitz continuous, then so are the cost-to-go functions.
In such a case, we can use cutting plane methods to build an under-approximation of the cost-to-go functions.
To be precise, for every \(\xi_t\in\Xi_t\), let \(V_t(\cdot;\xi_t)\) denote an affine function such that \(Q_t(x_{t-1};\xi_t)\ge V_t(x_{t-1};\xi_t)\) for all \(x_{t-1}\in\calX_{t-1}\).
Such an affine function is referred to as a linear cut for the value function, which is generated as follows.
Let \(\ulcQ_{t}\) denote an under-approximation of the cost-to-go function \(\calQ_{t}\) and \(\hat{x}_{t-1}\in\calX_{t-1}\) a feasible state.
For any fixed \(\xi_t\in\Xi_t\), 
the Lagrangian dual of the minimization problem~\eqref{eq:DR-MSCP-ValueFunction} at the state \(x_{t-1}=\hat{x}_{t-1}\) can be written as
\begin{equation}\label{eq:RecursionLagrangianDual}
    \sup_{\lambda_t\in\bbR^{d_{t-1}}}\inf_{\substack{x_t\in\calX_{t},\\z_t\in\bbR^{d_{t-1}}}}\quad f_t(z_t,x_t;\xi_t)+\ulcQ_{t}(x_t)+\bangle{\lambda_t}{\hat{x}_{t-1}-z_t}
\end{equation}
gives an affine function \(V_t(x_{t-1};\xi_t)\coloneqq \ubar{v}_t(\xi_t)+\bangle{\hat\lambda_t(\xi_t)}{x_{t-1}-\hat{x}_{t-1}}\), where we assume that an optimal dual solution \(\hat\lambda_t(\xi_t)\) of \eqref{eq:RecursionLagrangianDual} and the associated infimum value \(\ubar{v}_t(\xi_t)>-\infty\) are found. 
Then, by definition~\eqref{eq:DR-MSCP-ValueFunction} and weak duality, for every \(x_{t-1}\in\calX_{t-1}\) 
\begin{align}
    &Q_t(x_{t-1};\xi_t)\ge\sup_{\lambda_t\in\bbR^{d_{t-1}}}\inf_{\substack{x_t\in\calX_t,\\z_t\in\bbR^{d_{t-1}}}}f_t(z_t,x_t;\xi_t)+\calQ_{t}(x_t)+\bangle{\lambda_t}{x_{t-1}-z_t}\label{eq:LinearCutValidness}\\
    &\ge\inf_{\substack{x_t\in\calX_t,\\z_t\in\bbR^{d_{t-1}}}}f_t(z_t,x_t;\xi_t)+\ulcQ_{t}(x_t)+\bangle{\hat{\lambda}_t(\xi_t)}{\hat{x}_{t-1}-z_t}+\bangle{\hat{\lambda}_t(\xi_t)}{x_{t-1}-\hat{x}_{t-1}}\notag\\
    &=\ \ubar{v}_t(\xi_t)+\bangle{\hat{\lambda}_t(\xi_t)}{x_{t-1}-\hat{x}_{t-1}}.\notag
\end{align}
Therefore, \(V_t(\cdot;\xi_t)\) is a valid linear cut for the value function \(Q_t(\cdot;\xi_t)\).
Finally, we can aggregate the valid linear cuts to get an affine function
\begin{equation}\label{eq:AggregateLinearCut}
    \calV_{t-1}(x_{t-1})\coloneqq\bbE_{\xi_t\sim \hat{p}_t}V_t(x_{t-1};\xi_t),
\end{equation}
where \(\hat{p}_t\in\calP_t\) can be arbitrarily chosen.
The affine function \(\calV_{t-1}(\cdot)\) is a valid linear cut for the cost-to-go function \(\calQ_{t-1}(\cdot)\) because for any \(x_{t-1}\in\calX_{t-1}\),
\begin{align}
    \calQ_{t-1}(x_{t-1})-\calV_{t-1}(x_{t-1})
    &\ge\bbE_{\xi_t\sim\hat{p}_t}\Bigl[Q_t(x_{t-1};\xi_t)-V_t(x_{t-1};\xi_t)\Bigr]\ge0.\label{eq:AggregateLinearCutValidness}
\end{align}

The above procedure of generating valid linear cuts requires finding the supremum in the Lagrangian dual problem~\eqref{eq:RecursionLagrangianDual}.
If~\eqref{eq:DR-MSCP-Def} does not have relatively complete recourse (RCR), the supremum in~\eqref{eq:RecursionLagrangianDual} may not be attained.
Even with RCR, as we use some under-approximation \(\ulcQ_t\) in the dual problem~\eqref{eq:RecursionLagrangianDual}, the maximum may only be attained outside a ball with its radius growing with \(T\), as illustrated below. 
\begin{example}\label{ex:RecursionLipschitzConstant}
    Consider a deterministic problem (i.e., \(\abs{\Xi_t}=1\) for all \(t\in\calT'\)) defined by the recursions for \(t=T,\dots,2\)
    \begin{align*}
        \calQ_{t-1}(x_{t-1})\coloneqq\min_{y_t,x_t}\quad& 1+y_t+\calQ_{t}(x_t)\\
        \mathrm{s.t.}\quad &y_t\ge0,\ y_t\ge1-2x_{t-1},\ x_t\le x_{t-1}+\frac{1}{2},\ 0\le x_t\le 1.
    \end{align*}
    Note that for each stage \(t\in\calT'\), since \(x_{t-1}\in[0,1]\), we have a feasible solution \(x_{t}=1/2\) and \(y_t=0\), which implies that \(\calQ_{t}(x_{t})=T-t+\max\{0,1-2x_{t}\}\) for \(t\le T-1\).
    However, if we start our approximation with points \(x_t^0=0\) for all stages \(t\in\calT\), then the linear cut \(\calV_{t-1}(x_{t-1})=V_t(x_{t-1})\coloneqq \ubar{v}_t+\bangle{\hat{\lambda}_t}{x_t-x_t^0}\) can be generated from the following dual problem at stage \(t\in\calT'\):
    \begin{align*}
        \max_{\lambda_t\in\bbR}\inf_{z_t,y_t,x_t}\quad& 1+y_t+\lambda_t(0-z_t)+\ulcQ_{t}^0(x_t)\\
        \mathrm{s.t.}\quad &y_t\ge0,\ y_t\ge1-2z_t,\ x_t\le z_t+\frac{1}{2},\ 0\le x_t\le 1.
    \end{align*}
    The optimal dual solution is \(\hat{\lambda}_T=-2\) in stage \(T\).
    Thus the under-approximation of \(\calQ_{T-1}(x_{T-1})\) after adding the linear cut is \(\ulcQ_{T-1}^0(x_{T-1})=\calV_T(x_{T-1})=2-2x_{T-1}\).
    The dual problem in stage \(t=T-1\) now becomes
    \begin{align*}
        \max_{\lambda_t\in\bbR}\inf_{z_t,y_t,x_t}\quad& 1+y_t-\lambda_t z_t+(T-t+1)-c_tx_t\\
        \mathrm{s.t.}\quad &y_t\ge0,\ y_t\ge1-2z_t,\ x_t\le z_t+\frac{1}{2},\ 0\le x_t\le 1.
    \end{align*}
    with \(c_{T-1}=2\).
    Note that if \(\lambda_t>-c_t-2\), the inner minimum value would be strictly smaller than \(T-t+3-c_t/2\) (e.g., taking \(z_t=1/2\), \(y_t=0\), and \(x_t=1\)).
    Thus any optimal dual solution must satisfy \(\lambda_t\le -c_t-2\) as \(z_t=0\), \(y_t=1\), and \(x_t=1/2\) gives an inner minimum value \(T-t+3-c_t/2\).
    By repeating the argument recursively, it can be checked that the under-approximation \(\ulcQ_t^0(x_t)=\calV_t(x_t)=T-t+1-c_tx_t\) with \(c_t=2(T-t)\).
This implies that any optimal dual solution in stage \(t\) would lie outside the open ball of radius \(2(T-t+1)\), which is greater than 2, the Lipschitz constant of \(\calQ_t\) on \(\calX_t\).
We illustrate the example with Figure~\ref{fig:ExampleLipschitzGrowth} using \(T=3\).
\end{example}

\begin{figure}[htbp]
    \centering
    \begin{subfigure}{0.4\textwidth}
        \includegraphics[width=\textwidth]{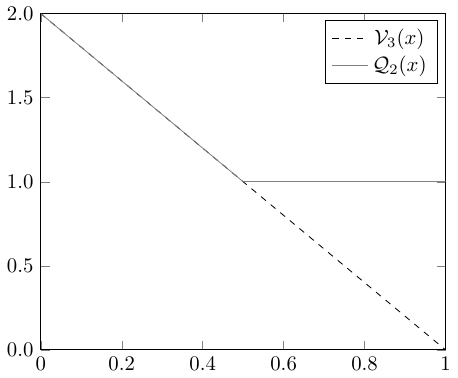}
        \caption{\(t=2\)}
    \end{subfigure}
    \hspace{5mm}
    \begin{subfigure}{0.4\textwidth}
        \includegraphics[width=\textwidth]{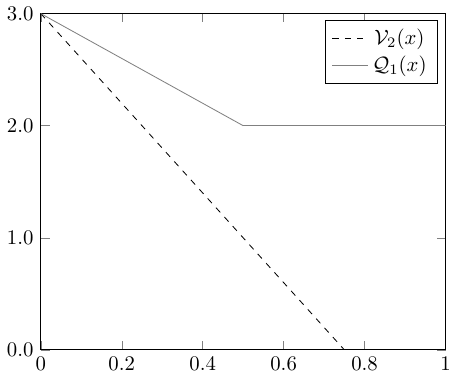}
        \caption{\(t=1\)}
    \end{subfigure}
    \caption{Cost-to-go Functions and First Iteration Linear Cuts for \(T=3\)}
    \label{fig:ExampleLipschitzGrowth}
\end{figure}

\subsection{Exactness of Dual Bounds}
\label{subsec:DualBoundExactness}

Example~\ref{ex:RecursionLipschitzConstant} suggests that the norm of optimal dual solutions for linear cut generation cannot be directly bounded by the Lipschitz constants of the value functions.
Intuitively speaking, larger norms of dual variables lead to steeper linear cuts~\eqref{eq:RecursionLagrangianDual}, which makes the under-approximation less useful at points away from the previously visited ones.
For this reason, we propose to enforce bounds on the dual variables for recursive approximations and investigate the exactness of these artificial dual bounds.
To begin with, let us fix algorithmic parameters \(M_t>0\) for stages \(t\in\calT'\) and define alternative value function recursively
\begin{equation}\label{eq:RegularizedValueFunction}
    Q_t^\Reg(x_{t-1};\xi_t)\coloneqq\max_{\nVert{\lambda_t}_*\le M_t}\inf_{\substack{x_t\in\calX_t,\\z_t\in\bbR^{d_{t-1}}}}f_t(x_{t-1},x_t;\xi_t)+\calQ_t^\Reg(x_t)+\bangle{\lambda_t}{x_{t-1}-z_t},
\end{equation}
with \(\calQ_T^\Reg(x_T)\equiv0\) on \(\calX_T\) and
\begin{equation}\label{eq:RegularizedCostToGo}
    \calQ_{t-1}^\Reg(x_{t-1})\coloneqq\sup_{p_t\in\calP_t}\bbE_{\xi_t\sim p_t} Q_t^\Reg(x_{t-1};\xi_t)
\end{equation}
for \(t=T,T-1,\dots,2\).
We call them (Lipschitz) regularized value functions and (Lipschitz) regularized cost-to-go functions, due to the following observation.
\begin{proposition}\label{prop:RegularizationValueFunction}
    For any \(t\in\calT'\) and \(\xi_t\in\Xi_t\), the function \(Q_t^\Reg(\cdot;\xi_t)\) is \(M_t\)-Lipschitz continuous, and \(Q_t^\Reg(x_{t-1};\xi_t)\le Q_t(x_{t-1};\xi_t)\) for all \(x_{t-1}\in\calX_{t-1}\).
    Moreover, we can write \(Q_t^\Reg(x_{t-1};\xi_t)\) as a value function
    \[
        Q_t^\Reg(x_{t-1};\xi_t)=\min_{\substack{x_t\in\calX_t,\\z_t\in\bbR^{d_{t-1}}}} f_t(z_t,x_t;\xi_t)+\calQ_t^\Reg(x_t)+M_t\nVert{x_{t-1}-z_t}.
    \]
\end{proposition}
\begin{proof}
    We begin our proof for the equation with the following observation
    \begin{align*}
    &\max_{\nVert{\lambda_t}_*\le M_t}\inf_{\substack{x_t\in\calX_t,z_t\in\bbR^{d_{t-1}}}}f_t(x_{t-1},x_t;\xi_t)+\calQ_t^\Reg(x_t)+\bangle{\lambda_t}{x_{t-1}-z_t}\\
    &=\inf_{\substack{x_t\in\calX_t,z_t\in\bbR^{d_{t-1}}}}\max_{\nVert{\lambda_t}_*\le M_t} f_t(x_{t-1},x_t;\xi_t)+\calQ_t^\Reg(x_t)+\bangle{\lambda_t}{x_{t-1}-z_t}\\
    &=\min_{\substack{x_t\in\calX_t,z_t\in\bbR^{d_{t-1}}}}f_t(x_{t-1},x_t;\xi_t)+\calQ_t^\Reg(x_t)+M_t\nVert{x_{t-1}-z_t}.
    \end{align*}
    Here, the first equality is due to Sion's minimax theorem~\cite{komiya1988elementary}, where the maximization is taken over a compact ball of radius \(M_t\).
    The second equality follows from the dual representation of the norm functions~\cite[A.1.6]{boyd2004convex}.
    The infimum is attained since the sum \(f_t(x_{t-1},x_t;\xi_t)+\calQ_t^\Reg(x_t)+M_t\nVert{x_t-z_t}\) is bounded from below by our nonnegativity assumption of \(f_t\).
    Now it can be checked recursively for \(t=T,\dots,2\) using the definition of \(\calQ_t^\Reg\) that \(Q_t^\Reg(x_{t-1};\xi_t)\le Q_t(x_{t-1};\xi_t)\) for any \(x_{t-1}\in\calX_{t-1}\) because \(z_t=x_{t-1}\) is a feasible solution to the minimization.
    
    It remains to show the \(M_t\)-Lipschitz continuity of \(Q_t^\Reg(\cdot;\xi_t)\).
    Pick any \(x^1,x^2\in\calX_{t-1}\) and let \(x_t^i,z_t^i\) denote the solutions in the minimization associated with \(x^i\) for \(i=1\) and \(2\), respectively.
    Then, as \(x_t^2,z_t^2\) are also feasible solutions for the minimization associated with \(x^1\), we have
    \begin{align*}
        Q_t^\Reg(x^1;\xi_t)-Q_t^\Reg(x^2;\xi_t)
        &\le f_t(z_t^2,x_t^2;\xi_t)+\calQ_{t}^\Reg(x_t^2)+M_t\nVert{x^1-z_t^2}\\
        &\quad-f_t(z_t^2,x_t^2;\xi_t)-\calQ_{t}^\Reg(x_t^2)-M_t\nVert{x^2-z_t^2}\\
        &=M_t\bigl(\nVert{x^1-z_t^2}-\nVert{x^2-z_t^2}\bigr)
        \le M_t\nVert{x^1-x^2}.
    \end{align*}
    Now we can repeat the argument with exchanged indices 1 and 2 and show that \(Q_t^\Reg(x^2;\xi_t)-Q_t^\Reg(x^1;\xi_t)\le M_t\nVert{x^1-x^2}\), which completes the proof.
\end{proof}

We can now generate a linear cut \(V_t(\cdot;\xi_t)\) using a dual solution \(\hat{\lambda}_t\) and its associated infimum value \(\ubar{v}_t>-\infty\) of the following dual problem (cf.~\eqref{eq:RecursionLagrangianDual})
\begin{equation}\label{eq:RecursionLagrangianBoundedDual}
    \max_{\nVert{\lambda_t}_*\le M_t}\inf_{\substack{x_t\in\calX_{t},\\z_t\in\bbR^{d_{t-1}}}}\quad f_t(z_t,x_t;\xi_t)+\ulcQ_{t}(x_t)+\bangle{\lambda_t}{\hat{x}_{t-1}-z_t}.
\end{equation}
As \(\ulcQ_t\) is an under-approximation of the regularized cost-to-go function \(\calQ_t^\Reg\), we can use the same argument as in~\eqref{eq:LinearCutValidness} to see the validness of \(V_t(\cdot;\xi_t)\) as a linear cut for \(Q_t^\Reg(\cdot;\xi_t)\), for all \(\xi_t\in\Xi_t\), and thus also the validness of \(\calV_{t-1}(\cdot)\) as an aggregate linear cut for \(\calQ_{t-1}^\Reg(\cdot)\).
The Lipschitz constants of \(V_t(\cdot;\xi_t)\) and \(\calV_{t-1}(\cdot)\) are bounded by \(M_t\).

The bounded dual problems~\eqref{eq:RecursionLagrangianBoundedDual} for linear cut generation, by Proposition~\ref{prop:RegularizationValueFunction}, are equivalent to replacing the \(Q_t\) and \(\calQ_t\) with the regularized ones, \(Q_t^\Reg\) and \(\calQ_t^\Reg\), respectively.
We show that such replacement does not compromise any feasibility or optimality.
The following lemma is well-known, due to Kirszbraun~\cite{kirszbraun1934zusammenziehende} and McShane~\cite{mcshane1934extension}.
We provide a proof for easy reference.
\begin{lemma}\label{lemma:ConvexInfimalConvolutionLipschitz}
    Let \(\calX\subseteq\bbR^d\) be a full-dimensional convex set and \(f:\bbR^d\to\bbR\) be a convex function.
    If the restriction \(f\vert_\calX\) is \(M\)-Lipschitz continuous, then we have \(f(x)=\inf_{z\in\bbR^d}\{f(z)+M\nVert{x-z}\}\) for any \(x\in\calX\).
\end{lemma}
\begin{proof}
    Assume for contradiction that for some \(x\in\calX\), there exist \(z\in\bbR^d\) and \(\epsilon>0\) such that \(f(x)>f(z)+(1+\epsilon)M\nVert{z-x}\).
    If \(x\) lies in the interior \(\mathrm{int}\calX\) of \(X\), then we can find \(z'=x+\delta(x-z)\in\calX\) for some \(\delta>0\).
    By convexity, \((f(z')-f(x))/\nVert{z'-x}\ge(f(x)-f(z))/\nVert{x-z}>(1+\epsilon)M\), which contradicts the \(M\)-Lipschitz continuity of \(f\vert_\calX\).

    Otherwise if \(x\notin\mathrm{int}\calX\), since \(\calX\) is full-dimensional, we can find \(x'\in\mathrm{int}\calX\) with \(\nVert{x'-x}\le \nVert{z-x}\cdot\epsilon/2\), so \(\nVert{x'-z}\le(1+\epsilon/2)\nVert{z-x}\).
    By the \(M\)-Lipschitz continuity of \(f\vert_\calX\), we have \(f(x')\ge f(x)-M\nVert{x-x'}>f(z)+(1+\epsilon/2)M\nVert{z-x}\).
    Therefore,
    \[
        \frac{f(x')-f(z)}{\nVert{x'-z}}>\frac{(1+\epsilon/2)M\nVert{z-x}}{(1+\epsilon/2)\nVert{z-x}}=M,
    \]
    which shows contradiction by the argument above since \(x'\in\mathrm{int}\calX\).
\end{proof}

\begin{proposition}\label{prop:LipschitzRegularizationExactness}
    Suppose that the state spaces \(\calX_t\) are full-dimensional.
    If the value function \(Q_t(\cdot;\xi_t)\) is \(M_t\)-Lipschitz continuous for any \(\xi_t\in\Xi_t\) and \(t\in\calT\), then we have \(Q_t^\Reg(x_{t-1};\xi_t)=Q_t(x_{t-1};\xi_t)\) for all \(x_{t-1}\in\calX_{t-1}\) and \(t\in\calT\).
\end{proposition}
\begin{proof}
    We prove the assertion by applying Lemma~\ref{lemma:ConvexInfimalConvolutionLipschitz} recursively for \(t=T,T-1,\dots,1\).
    Suppose that \(\calQ_t=\calQ_t^\Reg\) on \(\calX_t\) for some \(t\in\calT\), which holds trivially for \(t=T\).
    We see that \(Q_t(\cdot;\xi_t)=Q_t^\Reg(\cdot;\xi_t)\) everywhere on \(\calX_{t-1}\) for any \(\xi_t\in\Xi_t\) by Proposition~\ref{prop:RegularizationValueFunction}.
    Thus by definition~\eqref{eq:RegularizedCostToGo}, \(\calQ_{t-1}=\calQ_{t-1}^\Reg\) on \(\calX_{t-1}\).
\end{proof}

Proposition~\ref{prop:LipschitzRegularizationExactness} ensures the exactness of the bounded dual recursion~\eqref{eq:RegularizedValueFunction}, provided that the value functions \(Q_t(\cdot;\xi_t)\) are \(M_t\)-Lipschitz continuous for some known values \(M_t>0\).
We remark that the idea of bounding the dual variables, or equivalently adding a regularization term in the dual problems, could work more generally for multistage stochastic problems without RCR~\cite{zhang2022stochastic}.
While in general it is not easy to estimate the value of Lipschitz constants \(M_t\), updating the estimate \(\tilde{M}_t\) of \(M_t\) by a constant factor iteratively and resolving the DR-MCO problem will only add a logarithmic factor in the following complexity analysis.
We thus assume that \(M_t\) is known in the rest of the paper.

\section{Algorithms and Complexity Analysis}
\label{sec:Algorithms}
In this section, we first define single stage subproblem oracles (SSSO) for DR-MCO~\eqref{eq:DR-MSCP-Def}, based on which we define the notion of complexity of the algorithms.
A simple implementation of the SSSO is then discussed for the finitely supported problems.
We generalize the consecutive DDP (CDDP) algorithm to DR-MCO and more importantly, we introduce a new nonconsecutive DDP algorithm (NDDP) together with their complexity analysis.

\subsection{Single Stage Subproblem Oracles}
\label{subsec:SubproblemOracle}
A subproblem oracle is an oracle that gives a solution to the subproblem given the intermediate data generated by the algorithm.
The single stage subproblem oracles (SSSO) used in this paper solve an approximation of the problem defined by~\eqref{eq:RegularizedValueFunction} and~\eqref{eq:RegularizedCostToGo} for some stage \(t\in\calT\).
\begin{definition}[Initial stage subproblem oracle]
    \label{def:InitialStageOracle}
    Let \(\ulcQ_1,\olcQ_1:\calX_1\to\bbR\cup\{+\infty\}\) denote two lsc, convex functions, with \(\ulcQ_1(x_1)\le\calQ_1^\Reg(x_1)\le\olcQ_1(x_1)\) for any \(x_1\in\calX_1\).
    Consider the following subproblem for the first stage \(t=1\),
    \begin{equation}\label{eq:InitialStageOracle}
        \min_{x_1\in\calX_1}f_1(x_0,x_1;\xi_1)+\ulcQ_1(x_1).\tag{I}
    \end{equation}
    The initial stage subproblem oracle provides an optimal solution \(x_1\) to~\eqref{eq:InitialStageOracle} and calculates the approximation gap \(\gamma_1\coloneqq\olcQ_1(x_1)-\ulcQ_1(x_1)\).
    We thus define the subproblem oracle formally as a map \(\scrO_1:(\ulcQ_1,\olcQ_1)\mapsto(x_1;\gamma_1)\).
\end{definition}

\begin{definition}[Noninitial stage subproblem oracle]
    \label{def:NoninitialStageOracle}
    For any \(t\in\calT'\), let \(\ulcQ_t,\olcQ_t:\calX_t\to\bbR\cup\{+\infty\}\) denote two lsc, convex functions, with \(\ulcQ_t(x_t)\le\calQ_t^\Reg(x_t)\le\olcQ_t(x_t)\) for all \(x_t\in\calX_t\).
    Then given a feasible state \(x_{t-1}\in\calX_{t-1}\), the stage-\(t\) subproblem oracle provides a feasible state \(x_t\in\calX_t\), an \(M_t\)-Lipschitz continuous linear cut \(\calV_{t-1}(\,\cdot\,)\), and an over-estimate value \(v_{t-1}\) such that
    \begin{compactitem}
        \item they are valid, i.e., \(\calV_{t-1}(\cdot)\le\calQ_{t-1}^\Reg(\cdot)\) on \(\calX_{t-1}\) and \(v_{t-1}\ge\calQ_{t-1}^\Reg(x_{t-1})\);
        \item the gap is controlled, i.e., \(v_{t-1}-\calV_{t-1}(x_{t-1})\le\gamma_t\coloneqq\olcQ_t(x_t)-\ulcQ_t(x_t)\).
    \end{compactitem}
    We thus define the subproblem oracle formally as a map \(\scrO_t:(x_{t-1},\ulcQ_t,\olcQ_t)\mapsto(\calV_{t-1},v_{t-1},x_t;\gamma_t)\).
\end{definition}

The noninitial stage subproblem oracles are different from the initial stage subproblem oracle, in the sense that it does not necessarily provide any optimal solution to some optimization problem.
Instead, it provides some feasible state, which could be used for exploration of the following stages, a linear cut and an estimate value for updating the approximation in the previous stages.
We provide an illustration of the noninitial SSSO using finitely supported DR-MCO below.

\subsubsection{SSSO Implementation for Finitely Supported DR-MCO}
\label{subsec:SSSOImplementation}
We now propose a possible realization of the noninitial stage subproblem oracle based on the linear cut generation procedure discussed in Section~\ref{subsec:RecursionApproximation}.
Suppose \(\Xi_t=\{\hat{\xi}_{t,1},\dots,\hat{\xi}_{t,n_t}\}\) for each \(t\in\calT'\).
By solving the bounded Lagrangian dual problem~\eqref{eq:RecursionLagrangianBoundedDual} for each \(\hat{\xi}_{t,k}\), we can get primal solutions \(\hat{x}_{t,k}\), \(\hat{z}_{t,k}\), and linear cuts \(V_{t}(x_{t-1};\hat{\xi}_{t,k})\) for \(k=1,\dots,n_t\) that are all \(M_t\)-Lipschitz continuous.
We then aggregate them using \(\calV_{t-1}(\cdot)\coloneqq\sum_{k=1}^{n_t}\hat{p}_{t,k}V_t(\cdot;\hat{\xi}_{t,k})\), where \(\hat{p}_{t}\in\argmax_{p_{t}\in\calP_{t}}\sum_{k=1}^{n_t}p_{t,k}V_t(x_{t-1};\hat{\xi}_{t,k})\) is a maximizer at the given state \(x_{t-1}\).
It is clear that \(\calV_{t-1}(\cdot)\) is also \(M_t\)-Lipschitz continuous.
Next we solve the maximization problem
\begin{equation}\label{eq:NoninitialStageOracle-Overestimate}
    v_{t-1}\coloneqq\max_{p_t\in\calP_t}\sum_{k=1}^{n_t}p_{t,k}\cdot\bigl(f_t(x_{t-1},\hat{x}_{t,k};\hat{\xi}_{t,k})+\olcQ_t(\hat{x}_{t,k})+M_t\nVert{x_{t-1}-\hat{z}_{t,k}}\bigr).
\end{equation}
Note that by assumption \(\olcQ_t(x_t)\ge\calQ_t^\Reg(x_t)\) for all \(x_t\in\calX_t\), and \((\hat{x}_{t,k},\hat{z}_{t,k})\) is a feasible solution to the problem
\[
    \min_{\substack{x_t\in\calX_t,\\z_t\in\bbR^{d_{t-1}}}} f_t(z_t,x_t;\hat{\xi}_{t,k})+\olcQ_t(x_t)+M_t\nVert{x_{t-1}-z_t}\ge Q_t^\Reg(x_{t-1};\hat{\xi}_{t,k}).
\]
Thus the value \(v_{t-1}\ge\sum_{k=1}^{n_t}p_{t,k}Q^\Reg_t(x_{t-1};\hat{\xi}_{t,k})=\calQ_{t-1}^\Reg(x_{t-1})\) is a valid overestimation.

It remains to determine the feasible state \(x_t\) and calculate the gap \(\gamma_t\) that satisfy the second requirement in Definition~\ref{def:NoninitialStageOracle}.
Let \(\gamma_{t,k}\coloneqq\olcQ_t(\hat{x}_{t,k})-\ulcQ_t(\hat{x}_{t,k})\) for each \(k=1,\dots,n_t\).
We pick the index \(k^*\) corresponding to the largest gap \(\gamma_{t,k^*}\), and set \(x_t=\hat{x}_{t,k^*}\), \(\gamma_t=\gamma_{t,k^*}\).
Consequently, we have
\begin{align}
    v_{t-1}&=\max_{p_{t}\in\calP_{t}}\sum_{k=1}^{n_t}p_{t,k}\bigl(f_t(\hat{z}_{t,k},\hat{x}_{t,k};\hat{\xi}_{t,k})+\olcQ_{t}(\hat{x}_{t,k})+M_t\nVert{x_{t-1}-\hat{z}_{t,k}}\bigr)\notag\\
           &=\max_{p_{t}\in\calP_{t}}\sum_{k=1}^{n_t}p_{t,k}\bigl(f_t(\hat{z}_{t,k},\hat{x}_{t,k};\hat{\xi}_{t,k})+\gamma_{t,k}+\ulcQ_{t}(\hat{x}_{t,k})+M_t\nVert{x_{t-1}-\hat{z}_{t,k}}\bigr)\notag\\
           &\le \gamma_t+\max_{p_{t}\in\calP_{t}}\sum_{k=1}^{n_t}p_{t,k}\bigl(f_t(\hat{z}_{t,k},\hat{x}_{t,k};\hat{\xi}_{t,k})+\ulcQ_{t}(\hat{x}_{t,k})+M_{t}\nVert{x_{t-1}-\hat{z}_{t,k}}\bigr)\notag\\
    &= \gamma_t+\calV_{t-1}(x_{t-1}).\notag
\end{align}
We summarize the above implementation in Algorithm~\ref{alg:SubproblemOracle}.
In the special case of MSCO, the ambiguity set \(\calP_t\) is a singleton so the maximization problems in lines 7-8 are superficial where we simply need to aggregate the linear cuts \(V_t(x;\hat{\xi}_{t,k})\) and the over-estimate values \(\bar{v}_{t,k}\) according to the nominal probability weights.
Consequently, for MSCO, Algorithm~\ref{alg:SubproblemOracle} reduces to solving both the forward-step subproblem (for \(\hat{x}_{t,k}\) and \(\hat{z}_{t,k}\)) and the backward-step subproblem (for \(\hat{\lambda}_{t,k}\) and \(\ubar{v}_{t,k}\)) of the deterministic DDP algorithm in~\cite{zhang2022stochastic}.
\begin{algorithm}[ht]
    \caption{A Realization of Noninitial Stage Subproblem Oracle}
    \label{alg:SubproblemOracle}
    \begin{algorithmic}[1]
        \Require{\(\Xi_t=\{\hat{\xi}_{t,1},\dots,\hat{\xi}_{t,n_t}\}\) and \((x_{t-1},\ulcQ_t,\olcQ_t)\)} 
        \Ensure{\((\calV_{t-1},v_{t-1},x_t;\gamma_t)\) satisfying Definition~\ref{def:NoninitialStageOracle}}
        \For{\(k=1,\dots,n_t\)}
        \State{solve the bounded Lagrangian dual problem~\eqref{eq:RecursionLagrangianBoundedDual} associated with \(\hat{\xi}_{t,k}\)} 
        \State{collect the primal-dual solution pair \((\hat{x}_{t,k},\hat{z}_{t,k};\hat{\lambda}_{t,k})\) and the optimal value \(\ubar{v}_{t,k}\)}
        \State{define \(V_t(x;\hat{\xi}_{t,k})\coloneqq\ubar{v}_{t,k}+\bangle{\hat{\lambda}_{t,k}}{x-x_{t-1}}\)}
        \State{calculate \(\gamma_{t,k}\coloneqq\olcQ_t(\hat{x}_{t,k})-\ulcQ_t(\hat{x}_{t,k})\)}
        \State{let \(\bar{v}_{t,k}:=f_t(x_{t-1},\hat{x}_{t,k};\hat{\xi}_{t,k})+\olcQ_t(\hat{x}_{t,k})+M_t\nVert{x_{t-1}-\hat{z}_{t,k}}\)}
        \EndFor
        \State{construct \(\calV_{t-1}(x)\coloneqq\sum_{k=1}^{n_t}\hat{p}_{t,k}V_{t}(x;\hat{\xi}_{t,k})\) with \(\hat{p}_{t}\in\argmax_{p_{t}\in\calP_{t}}\sum_{k=1}^{n_t}p_{t,k}\ubar{v}_{t,k}\)}
        \State{calculate \(v_{t-1}\coloneqq\max_{p_{t}\in\calP_{t}}\sum_{k=1}^{n_t}p_{t,k}\bar{v}_{t,k}\)}
        \State{find \(k^*\) such that \(\gamma_{t,k^*}\ge\gamma_{t,k}\) for all \(k=1,\dots,n_t\) and set \(x_t\coloneqq \hat{x}_{t,k^*},\gamma_t\coloneqq\gamma_{t,k^*}\)}
    \end{algorithmic}
\end{algorithm}

We remark that Algorithm~\ref{alg:SubproblemOracle} is not the only way to realize the SSSO in Definition~\ref{def:NoninitialStageOracle}.
For example, it is discussed in \cite{georghiou_robust_2019} that a polyhedral single stage subproblem of MRCO can be reformulated as mixed-integer linear optimization, which may then be solved by branch-and-bound type algorithms.
Therefore, the introduction of SSSO may benefit our discussion by avoiding restrictions of solution methods in each stage, even when the uncertainty sets are finite.
Besides, with SSSO, the complexity analysis may better reflect the computation time as the for-loop in Algorithm~\ref{alg:SubproblemOracle} can be easily parallelized.
We also show in the next section that SSSO enables us to introduce a nonconsecutive dual dynamic programming algorithm.

\subsection{Dual Dynamic Programming Algorithms}
\label{subsec:DualDPAlgorithms}

With the subproblem oracles, we first extend the consecutive dual dynamic programming (CDDP) algorithm to the DR-MCO setting. 
Then we focus on a new nonconsecutive dual dynamic programming (NDDP) algorithm, which can achieve an improved subproblem oracle complexity over the CDDP, as will be shown in Section~\ref{subsec:ComplexityUB} 
Both CDDP and NDDP algorithms produce deterministic upper bounds for exploration and termination. 

\subsubsection{Consecutive Dual Dynamic Programming}
To ease the notation, we use \(\conv\{h_1,h_2\}\) to denote the function corresponding to the closed convex hull of the epigraphs of functions \(h_1\) and \(h_2\).
More precisely, we define
\[
    \conv\{h_1,h_2\}(x)\coloneqq \sup_{\lambda}\inf_{z}\left\{\min\{h_1(z),h_2(z)\}+\bangle{\lambda}{x-z}\right\}.
\]
Note that if \(h_1,h_2\) are further polyhedral and proper, then by linear optimization strong duality, 
the value of the function \(\conv\{h_1,h_2\}\) at any $x\in\bbR^n$ is given as
\[
    \conv\{h_1,h_2\}(x)=\min\left\{
    v\in\bbR:
    \begin{aligned}
    &\exists\;z_i\in\bbR^n,v_i\ge h_i(z_i),\text{ and }\mu_i\in\bbR_{\ge0},\ i=1,2,\\
    &\mathrm{s.t.}\,\mu_1+\mu_2=1,x=\mu_1z_1+\mu_2z_2,v=\mu_1v_1+\mu_2v_2
    \end{aligned}
    \right\}. 
\]
\begin{algorithm}[ht]
    \caption{Consecutive Dual Dynamic Programming Algorithm}
    \label{alg:ConsecutiveDualDP}
    \begin{algorithmic}[1]
        \Require{subproblem oracles \(\scrO_t\) for \(t\in\calT\), target optimality gap \(\epsilon>0\)} 
        \Ensure{an \(\epsilon\)-optimal first stage solution \(x_1^*\)}
        \State{initialize: \(\ulcQ_t^0\leftarrow 0,\olcQ_t^0\leftarrow +\infty,t\in\calT\backslash\{T\}\); \(\ulcQ_T^j,\olcQ_T^j\leftarrow0,j\in\bbN\); \(i\leftarrow 1\)}
        \State{evaluate \((x_1^1;\gamma_1^1)\leftarrow\scrO_1(\ulcQ_1^0,\olcQ_1^0)\)}
        \State{set \(\LB\leftarrow f_1(x_0,x_1^1;\xi_1),\ \UB\leftarrow +\infty\)}
        \While{\(\UB-\LB >\epsilon\)}
        \For{\(t=2,\dots,T\)}
        \State{evaluate \((\calV_{t-1}^{i},v_{t-1}^{i},x_t^i;\gamma_t^i)=\scrO_t(x_{t-1}^i,\ulcQ_t^{i-1},\olcQ_t^{i-1})\)}
        \Comment{forward step}
        \EndFor
        \For{\(t=T,\dots,2\)}
        \State{update \(\ulcQ_{t-1}^i(x)\leftarrow \max\{\ulcQ_{t-1}^{i-1}(x),\calV_{t-1}^i(x)\}\)}
        \Comment{backward step}
        \State{update \(\olcQ_{t-1}^i(x)\leftarrow \conv\{\olcQ_{t-1}^{i-1}(x),v_{t-1}^i+M_{t}\nVert{x-x_{t-1}^i}\}\)}
        \EndFor
        \State{evaluate \((x_1^{i+1};\gamma_1^{i+1})\leftarrow\scrO_1(\ulcQ_1^i,\olcQ_1^i)\)}
        \Comment{initial stage step}
        \State{update \(\LB\leftarrow f_1(x_0,x_1^{i+1};\xi_1)+\ulcQ_1^i(x_1^{i+1})\)}
        \State{update \(\UB'\leftarrow f_1(x_0,x_1^{i+1};\xi_1)+\olcQ_1^i(x_1^{i+1})\)}
        \If{\(\UB'<\UB\)}
        \State{set \(x_1^*\leftarrow x_1^{i+1}\),  \(\UB\leftarrow\UB'\)}
        \EndIf
        \State{update \(i\leftarrow i+1\)}
        \EndWhile
    \end{algorithmic}
\end{algorithm}

For each iteration \(i\in\bbN\), the main loop of CDDP (Algorithm~\ref{alg:ConsecutiveDualDP}) consists of three steps.
The forward step uses the state \(x_{t-1}^i\) in the previous stage and the approximations \(\ulcQ_{t}^{i-1}\) and \(\olcQ_{t}^{i-1}\) to produce a new state \(x_t^i\).
Then the backward step at stage \(t\) uses the cut \(\calV_{t-1}^i(x)\) and the value \(v_{t-1}^i\) to update the approximations \(\ulcQ_{t-1}^i,\olcQ_{t-1}^i\) in its precedent stage \(t-1\).
The initial stage step produces a first stage solution \(x_1^{i+1}\) and updates the lower and upper bounds.

We next discuss the correctness of Algorithm~\ref{alg:ConsecutiveDualDP}, i.e., the returned solution \(x_1^*\) is \(\epsilon\)-optimal as defined in~\eqref{eq:DR-MSCP-NearOptimalSolution}, while leaving the finiteness proof to Section~\ref{subsec:ComplexityUB}.
From the termination of the while-loop, it suffices to show that the approximations are valid \(\ulcQ_t^i(x)\le\calQ_t^\Reg(x)\le\olcQ_t^i(x)\) for each \(t\in\calT\) and \(i\in\bbN\).
The first inequality follows from the validness of linear cuts \(\calV_t^i\) (Definition~\ref{def:NoninitialStageOracle}).
The second inequality is due to the \(M_{t}\)-Lipschitz continuity of \(\calQ_{t-1}^\Reg\) (Proposition~\ref{prop:RegularizationValueFunction}).
In particular, the inequality \(v_{t-1}^i\ge\calQ_{t-1}^\Reg(x_{t-1}^i)\) implies \(v_{t-1}^i+M_t\nVert{x-x_{t-1}^i}\ge\calQ_t^\Reg(x)\) for all \(x\in\calX_{t-1}\).
Given that \(\olcQ_{t-1}^{i-1}(x)\ge\calQ_{t-1}^\Reg(x)\) for \(x\in\calX_{t-1}\), which is obviously true for \(i=1\), we conclude that
\begin{equation*}
    \min\{\olcQ_{t-1}^{i-1}(x),v_{t-1}^i+M_{t}\Vert{x-x_{t-1}^i}\Vert\}\ge\calQ_{t-1}^\Reg(x),\quad\forall\,x\in\calX_{t-1}.
\end{equation*}
By taking the closed convex hull of the epigraphs on both sides, we have shown that \(\olcQ_{t-1}^i(x)\ge\calQ_{t-1}^\Reg(x)\) for all \(x\in\calX_{t-1}\).
The above argument shows that for all \(i\in\bbN\), the approximations are valid, which then implies the correctness of the algorithm.

We comment that the linear cut \(\calV_{t-1}^{i}\) and the over-estimate value \(v_{t-1}^{i}\) are generated using only the information in the previous iteration \(i-1\).
In fact, the subproblem oracles can be re-evaluated in the backward steps to produce tighter approximations.
We simply keep the CDDP algorithm in its current form because it is already sufficient for us to prove its complexity bounds, and hence its convergence.
We next propose a nonconsecutive version, NDDP, that conducts more efficient approximation updates.
\begin{algorithm}[ht]
    \caption{Nonconsecutive Dual Dynamic Programming (NDDP) Algorithm}
    \label{alg:NonconsecutiveDualDP}
    \begin{algorithmic}[1]
        \Require{subproblem oracles \(\scrO_t\) for \(t\in\calT\), opt.\ and approx.\ gaps \(\epsilon=\delta_1>\cdots>\delta_T=0\)} 
        \Ensure{an \(\epsilon\)-optimal first stage solution \(x_1^*\)}
        \State{initialize: \(\ulcQ_t^0\leftarrow 0,\olcQ_t^0\leftarrow +\infty,t\in\calT\backslash\{T\}\); \(\ulcQ_T^j,\olcQ_T^j\leftarrow0,j\in\bbN\); \(i_t\leftarrow 0,t\in\calT\)}
        \State{set \(\LB\leftarrow 0,\ \UB\leftarrow +\infty\), \(t\leftarrow 1\)}
        \Loop
        \State{update \(i_t\leftarrow i_t+1\)}
        \If{\(t=1\)}
        \State{evaluate \((x_1^{i_1};\gamma_1^{i_1})\leftarrow\scrO_1(\ulcQ_1^{i_1},\olcQ_1^{i_1})\)}
        \Comment{initial stage step}
        \State{update \(\LB\leftarrow f_1(x_0,x_1^{i_1};\xi_1)+\ulcQ_1^{i_1}(x_1^{i_1})\)}
        \State{update \(\UB'\leftarrow f_1(x_0,x_1^{i_1};\xi_1)+\olcQ_1^{i_1}(x_1^{i_1})\)}
        \If{\(\UB'<\UB\)}
        \State{set \(x_1^*\leftarrow x_1^{i_1}\),  \(\UB\leftarrow\UB'\)}
        \EndIf
        \If{\(\UB-\LB\le\epsilon\)}
        \State{\bfseries{break the loop}}
        \EndIf
        \State{maintain \(\ulcQ_2^{i_2+1}(x)\leftarrow\ulcQ_2^{i_2}(x),\;\olcQ_2^{i_2+1}(x)\leftarrow\olcQ_2^{i_2}(x)\)}
        \State{set \(t\leftarrow t+1\)}
        \Else
        \State{evaluate \((\calV_{t-1}^{i_t},v_{t-1}^{i_t},x_t^{i_t};\gamma_t^{i_t})\leftarrow\scrO_t(x_{t-1}^{i_{t-1}},\ulcQ_t^{i_t},\olcQ_t^{i_t})\)}
        \Comment{noninitial stage step}
        \If{\(t<T\) \textbf{and} \(\gamma_t^{i_t}>\delta_t\)}
        \State{maintain \(\ulcQ_{t+1}^{i_{t+1}+1}(x)\leftarrow\ulcQ_{t+1}^{i_{t+1}}(x),\;\olcQ_{t+1}^{i_{t+1}+1}(x)\leftarrow\olcQ_{t+1}^{i_{t+1}}(x)\)}
        \State{set \(t\leftarrow t+1\)}
        \Else
        \State{update \(\ulcQ_{t-1}^{i_{t-1}+1}(x)\leftarrow \max\{\ulcQ_{t-1}^{i_{t-1}}(x),\calV_{t-1}^{i_t}(x)\}\)}
        \State{update \(\olcQ_{t-1}^{i_{t-1}+1}(x)\leftarrow \conv\{\olcQ_{t-1}^{i_{t-1}}(x),v_{t-1}^{i_t}+M_{t-1}\Vert{x-x_{t-1}^{i_{t-1}}}\Vert\}\)}
        \State{set \(t\leftarrow t-1\)}
        \EndIf
        \EndIf
        \EndLoop
    \end{algorithmic}
\end{algorithm}

\subsubsection{Nonconsecutive Dual Dynamic Programming}

We provide the description of NDDP in Algorithm~\ref{alg:NonconsecutiveDualDP}.
To start the algorithm, it requires an additionally chosen vector of approximation gaps \(\delta\coloneqq(\delta_t)_{t=1}^T\) such that \(\epsilon=\delta_1>\delta_2>\cdots>\delta_T=0\), compared with the CDDP algorithm.
These predetermined approximation gaps serve as criteria at stage \(t\) for deciding the next stage to be solved: the precedent stage \(t-1\) or the subsequent one \(t+1\).
If the algorithm decides to proceed to the subsequent stage \(t+1\), then the current state \(x_t\) is used; otherwise the generated linear cut \(\calV_{t-1}\) and over-estimate value \(v_{t-1}\) are used for updating the approximations.
The above argument of validness of approximations imply that \(\ulcQ_t^{i_t}(x)\le\calQ_t^\Reg(x)\le\olcQ_t^{i_t}(x)\) for all \(x\in\calX_t\) holds for any stage \(t\in\calT\) and any index \(i_t\in\bbN\).
Therefore, when NDDP terminates, the returned solution \(x_1^*\) is indeed \(\epsilon\)-optimal.

\subsection{Complexity Upper Bounds}
\label{subsec:ComplexityUB}

In this section, we provide a complexity analysis for the proposed CDDP and NDDP algorithms, which implies that both algorithms terminate in finite time for any finite \(M_2,\dots,M_T>0\) and \(\epsilon>0\).
Our goal is to derive an upper bound on the total number of subproblem oracle evaluations before the termination of the algorithm.
To begin with, let \(\calJ_t,t>1\), denote the set of pairs of indices \((i_{t-1},i_t)\) such that the noninitial stage subproblem oracle is evaluated at the \(i_t\)-th time at the state \(x_{t-1}^{i_{t-1}}\), i.e., \((\calV_{t-1}^{i_t},v_{t-1}^{i_t},x_t^{i_t};\gamma_t^{i_t})=\scrO_t(x_{t-1}^{i_{t-1}},\ulcQ_t^{i_t},\olcQ_t^{i_t})\).
For the CDDP algorithm, all stages share the same iteration index \(i_t=i\), so \(\calJ_t=\{(i,i):i\in\bbN\}\) for all \(t>1\).
We define the following sets of indices for each \(t\in\calT\setminus\{T\}\):
\begin{equation}
    \calI_t(\delta)\coloneqq\left\{i_t\in\bbN:\gamma_t^{i_t}>\delta_t\text{ and } \gamma_{t+1}^{i_{t+1}}\le\delta_{t+1},\ (i_t,i_{t+1})\in\calJ_{t+1}\right\}.
\end{equation}
Here, for NDDP algorithm, \(\delta\) is the given approximation gap vector, while for CDDP algorithm, \(\delta=(\delta_t)_{t=1}^T\) can be any vector satisfying \(\epsilon=\delta_1>\delta_2>\cdots>\delta_T=0\) for the purpose of analysis, since it is not required for the CDDP algorithm.
We adopt the convention that the gap for the last stage \(\gamma_T^{i_T}\equiv0\) such that \(i_{T-1}\in\calI_{T-1}(\delta)\) if and only if \(\gamma_{T-1}^{i_{T-1}}>\delta_{T-1}\) and \((i_{T-1},i_T)\in\calJ_T\).
In plain words, \(\calI_t(\delta)\) contains indices \(i_t\) such that the state \(x_t^{i_t}\) has an approximation gap greater than \(\delta_t\) at \(i_t\)-th evaluation, but the \(i_{t+1}\)-th evaluation of the stage-\((t+1)\) oracle finds a solution with approximation gap smaller than \(\delta_{t+1}\) for some \(i_{t+1}\in\bbN\).
Our definition of \(\calI_t(\delta)\) in the case of CDDP is analogous to the definition used for deterministic DDP algorithms for MSCO problems in~\cite{zhang2022stochastic}.
However, it is more flexible as it allows nonconsecutive visits of consecutive stages, which as we will see below, helps us to establish a better complexity upper bound on the NDDP algorithm.
\begin{lemma}\label{lemma:FiniteIndexSet}
    For stage any \(t<T\), suppose the state space \(\calX_t\subset\bbR^{d_t}\) is contained in a ball with diameter \(D_t>0\).
    Then,
    \begin{equation}\label{eq:FiniteIndexSet}
        \abs{\calI_t(\delta)}\le\left(1+\frac{2M_{t+1} D_t}{\delta_t-\delta_{t+1}}\right)^{d_t}.
    \end{equation}
\end{lemma}
\begin{proof}
    We claim that for any \(j,k\in\calI_t\), \(j\neq k\), it holds that \(\nVert{x_t^j-x_t^k}>(\delta_t-\delta_{t+1})/(2M_{t+1})\).
    Assume for contradiction that \(\nVert{x_t^j-x_t^k}\le(\delta_t-\delta_{t+1})/(2M_{t+1})\) for some \(j<k\), \(j,k\in\calI_t(\delta)\).
    By definition of \(\calI_t(\delta)\), the \((t+1)\)-th subproblem oracle is evaluated at the state \(x_t^j\), and in both the CDDP and the NDDP algorithms, the approximations \(\ulcQ_t^j\) and \(\olcQ_t^j\) are updated since \(\gamma_{t+1}^{i_{t+1}}\le\delta_{t+1}\) for some \(i_{t+1}\in\bbN\) with \((j,i_{t+1})\in\calJ_{t+1}\).
    Then by Definition~\ref{def:NoninitialStageOracle} of the noninitial stage subproblem oracle, we have \(\olcQ_t^{j+1}(x_t^j)-\ulcQ_t^{j+1}(x_t^j)\le \delta_{t+1}\).
    Note that for any point \(x\in\calX_t\) with \(\nVert{x-x_t^j}\le(\delta_t-\delta_{t+1})/(2M_{t+1})\), we have \(\olcQ_t^{j+1}(x)-\ulcQ_t^{j+1}(x)\le \delta_{t}\) because of the \(M_{t+1}\)-Lipschitz continuity of the approximations.
    Since \(\olcQ_t^{k}(x)\le\olcQ_t^{j+1}(x)\) and \(\ulcQ_t^{k}(x)\ge\ulcQ_t^{j+1}(x)\), by setting \(x=x_t^k\), we see a contradiction with the assumption that \(k\in\calI_t(\delta)\), which proves the claim.

    To ease the notation, let \(r_t\coloneqq(\delta_t-\delta_{t+1})/(2M_{t+1})\) denote the radius of the \(d_t\)-dimensional balls \(\calB^{d_t}(x_t^j;r_t)\) centered at \(x_t^j\) for \(j\in\calI_t(\delta)\), and let \(\calB_t\supseteq\calX_t\) denote a ball with diameter \(D_t\).
    From the above claim, we know that \(x_t^k\notin\calB^{d_t}(x_t^j;r_t)\) for any \(j,k\in\calI_t(\delta)\) with \(j<k\).
    In other words, the smaller balls \(\calB^{d_t}(x_t^j;r_t/2)\) are disjoint.
    Meanwhile, note that each of these smaller balls satisfies \(\calB^{d_t}(x_t^j;r_t/2)\subset\calB_t+\calB^{d_t}(0;r_t/2)\) (the Minkowski sum in the Euclidean space \(\bbR^{d_t}\)).
    Therefore, the volumes satisfy the relation
    \begin{equation*}
     \Vol\left(\textstyle\bigcup_{j\in\calI_t(\delta)}\calB^{d_t}(x_t^j;r_t/2)\right)
        =\abs{\calI_t(\delta)}\cdot\Vol\calB^{d_t}(0;r_t/2)
        \le\Vol\left(\calB_t+\calB^{d_t}(0;r_t/2)\right),
    \end{equation*}
    which implies that
    \begin{equation*}
        \abs{\calI_t(\delta)}\le\frac{\Vol(\calB_t+\calB^{d_t}(0;r_t/2))}{\Vol\calB^{d_t}(0;r_t/2)}
        =\left(\frac{D_t/2+r_t/2}{r_t/2}\right)^{d_t}
        =\left(1+\frac{2M_{t+1}D_t}{\delta_t-\delta_{t+1}}\right)^{d_t}.
    \end{equation*}
\end{proof}

We prove the following complexity upper bounds for the CDDP algorithm (Theorem \ref{thm:CDDPComplexityBound}) and the NDDP algorithm (Theorem \ref{thm:NDDPComplexityBound}).
\begin{theorem}\label{thm:CDDPComplexityBound}
    Suppose the state spaces \(\calX_t\subset\bbR^{d_t}\) are contained in balls, each with diameter \(D_t>0\).
    Then for the CDDP algorithm (Algorithm~\ref{alg:ConsecutiveDualDP}), the total number of subproblem oracle evaluations \(\NumEval_\CDDP\) before termination is bounded by
    \begin{equation*}
        \NumEval_\CDDP\le 1+T\cdot\inf_\delta\left\{\sum_{t=1}^{T-1}\left(1+\frac{2M_{t+1}D_t}{\delta_t-\delta_{t+1}}\right)^{d_t}:\epsilon=\delta_1>\delta_2>\cdots>\delta_T=0\right\}.
    \end{equation*}
\end{theorem}
\begin{proof}
    We prove by showing that for any approximation gap vector \(\delta\) satisfying \(\epsilon=\delta_1>\delta_2>\cdots>\delta_T=0\), the largest iteration index \(i\) is bounded by
    \begin{equation}\label{eq:CDDPIterationBound}
        i\le\sum_{t=1}^{T-1}\left(1+\frac{2M_{t+1}D_t}{\delta_t-\delta_{t+1}}\right)^{d_t}.
    \end{equation}
    We claim that each iteration \(i\in\bbN\) must lie in either of the following two cases:
    \begin{compactenum}
        \item the initial stage step has \(\gamma_1^{i}\le\epsilon\); or
        \item the \(i\)-th forward step is in the index set \(i\in\calI_t(\delta)\) for some stage \(t<T\).
    \end{compactenum}
    To see the claim, suppose that the iteration \(i\in\bbN\) is not in the first case. 
    Then we have \(\gamma_1^i>\epsilon\) and by convention \(\gamma_T^i=0\le\delta_T\). 
    Therefore, there exists a stage \(t<T\) such that \(\gamma_t^i>\delta_t\) while \(\gamma_{t+1}^i\le\delta_{t+1}\), which is the second case.
    Note that when the first case happens, we have \(\UB-\LB\le\gamma_1^i\le\epsilon\) and thus the CDDP algorithm terminates.
    By Lemma~\ref{lemma:FiniteIndexSet}, the second case can only happen at most \(\sum_{t=1}^{T-1}\abs{\calI_t(\delta)}\) times, proving the bound~\eqref{eq:CDDPIterationBound}.
    The theorem then follows from the fact that in each CDDP iteration, the subproblem oracle is evaluated \(T+1\) times with the additional evaluation of the initial stage subproblem oracle for checking the termination criterion.
\end{proof}

\begin{remark}
    When \(d_t=d\), \(M_t=M\), and \(D_t=D\) for all \(t=1,\dots,T-1\), the bound in Theorem~\ref{thm:CDDPComplexityBound} can be simplified nicely.
    Note that the function
    \[
        \sum_{t=1}^{T-1}\left(1+\frac{2MD}{\epsilon\cdot\sigma_t}\right)^d
    \]
    is strictly convex and symmetric under permutation on the simplex \(\Delta^{T-1}:=\{\sigma=(\sigma_1,\dots,\sigma_{T-1})\in\bbR^{T-1}:\sum_{t=1}^{T-1}\sigma_t=1\}\).
    Therefore, it has a unique optimal solution \(\sigma=(\frac{1}{T-1},\dots,\frac{1}{T-1})\), which implies the infimum in~\eqref{eq:CDDPIterationBound} is attained by \(\delta_t=\frac{T-t}{T-1}\epsilon\), \(t\in\calT\).
    In this case, the CDDP complexity upper bound in Theorem~\ref{thm:CDDPComplexityBound} becomes
    \[
        i\le 1+T(T-1)\left(1+\frac{2MD(T-1)}{\epsilon}\right)^d.
    \]
\end{remark}

\begin{theorem}\label{thm:NDDPComplexityBound}
    Suppose the state spaces \(\calX_t\subset\bbR^{d_t}\) are contained in balls, each with diameter \(D_t>0\).
    Then, for the NDDP algorithm (Algorithm~\ref{alg:NonconsecutiveDualDP}) with the predetermined approximation gap vector \((\delta_t)_{t=1}^{T}\) satisfying \(\epsilon=\delta_1>\delta_2>\cdots>\delta_T=0\), 
    the total number of subproblem oracle evaluations \(\NumEval_\NDDP\) before termination is bounded by
    \begin{equation*}
        \NumEval_\NDDP\le 1+2\cdot\sum_{t=1}^{T-1}\left(1+\frac{2M_{t+1}D_t}{\delta_t-\delta_{t+1}}\right)^{d_t}.
    \end{equation*}
\end{theorem}
\begin{proof}
    For the NDDP algorithm, each time when it decides to go back to the precedent stage \(t\leftarrow t-1\), we must have \(\gamma_t^{i_t}\le\delta_t\) while \(\gamma_{t-1}^{i_{t-1}}>\delta_{t-1}\) for some \((i_{t-1},i_t)\in\calJ_t\).
    In this case, we have by definition that \(i_{t-1}\in\calI_{t-1}(\delta)\).
    By Lemma~\ref{lemma:FiniteIndexSet}, such ``going back'' steps can only happen at most \(\sum_{t=1}^{T-1}\abs{\calI_t(\delta)}\) times.
    The total number of evaluations is then bounded by two times the number of ``going back'' steps due to the corresponding ``going forward'' steps, plus an additional evaluation of the initial stage oracle for checking the termination criterion.
\end{proof}

We compare the complexity bounds of the two algorithms.
If we fix the approximation gap vector \(\delta\) to be \(\delta=(1,\frac{T-2}{T-1},\dots,\frac{1}{T-1})\cdot\epsilon\) in both bounds from Theorem~\ref{thm:CDDPComplexityBound} and Theorem~\ref{thm:NDDPComplexityBound}, then we can see that the bound on \(\NumEval_\CDDP\) is \(T/2\) times the bound on \(\NumEval_\NDDP\).
Both bounds are polynomial in \(T\) for fixed dimensions \(d_t\).
Moreover, in the special case of MSCO, we can compare both upper bounds against those known in the literature~\cite{lan2022complexity,zhang2022stochastic}. 
In particular, the bound on \(\NumEval_\CDDP\) is the same as \(T\) times the bound on the number of iterations needed for the deterministic DDP algorithm to converge in~\cite{zhang2022stochastic} (with the same dependency on the diameter \(D_t\) and Lipschitz constant \(M_{t+1}\)), which is futher equivalent to the bound in~\cite[Theorem 2]{lan2022complexity} when the MSCO is not discounted.
This matches our intuition that CDDP is a natural extension of the deterministic DDP algorithm to DR-MCO problems, in which we executes \(T\) subproblem oracles consecutively within each iteration.
Our bound on \(\NumEval_\NDDP\) then shows an improvement over the best known bound (in terms of \(T\)) in the literature for MSCO.
As we will see below, it is in fact optimal if we allow the target optimality gap to grow with \(T\), which is reasonable in many practical problems.

Recall that the local cost functions \(f_1,\dots,f_T\) are assumed to be nonnegative.
Given any \(\alpha\ge0\), we say \(x_1^*\in\calX_1\) is an \(\alpha\)-relative optimal solution if 
\begin{equation*}f_1(x_0,x_1^*;\xi_1)+Q_1(x_1^*)\le(1+\alpha)\left(\min_{x_1\in\calX_1}f_1(x_0,x_1;\xi_1)+\calQ_1(x_1)\right).
\end{equation*}
An \(\epsilon\)-optimal first stage solution is \(\alpha\)-relative optimal if 
\[\epsilon\le\alpha\left(\min_{x_1\in\calX_1}f_1(x_0,x_1;\xi_1)+\calQ_1(x_1)\right).
\]
We provide below an important simplification of the above complexity bounds when considering relative optimality gaps.
\begin{corollary}\label{cor:DDPComplexityRelativeGap}
    Suppose that all the state spaces have the same dimension \(d_t=d\) and are bounded by a common diameter \(D_t\le D\), and let \(M\coloneqq\max\{M_t:t\in\calT'\}\).
    If for each stage \(t\in\calT\), the local cost functions are uniformly bounded from below by \(C>0\) for all feasible solutions \(x_{t-1},x_t\) and uncertainty outcomes \(\xi_t\), i.e., \(f_t(x_{t-1},x_t;\xi_t)\ge C\), 
    then the total number of subproblem oracle evaluations before achieving an \(\alpha\)-relative optimal solution \(x_1^*\) for CDDP and NDDP are upper bounded respectively by
    \begin{align*}
        &\NumEval_\CDDP\le 1+T(T-1)\left(1+\frac{2MD}{\alpha C}\right)^d,\\
        &\NumEval_\NDDP\le 1+2(T-1)\left(1+\frac{2MD}{\alpha C}\right)^d.
    \end{align*}
\end{corollary}
\begin{proof}
    By assumption, \(\min_{x_1\in\calX_1}f_1(x_0,x_1;\xi_1)+\calQ_1(x_1)\ge CT\).
    Thus the assertion is proved if \(x_1^*\) is \(\epsilon\)-optimal with \(\epsilon=\alpha C(T-1)<\alpha CT\).    
    We can apply Theorems~\ref{thm:CDDPComplexityBound} and~\ref{thm:NDDPComplexityBound} by setting \(\delta_t=(T-t)\alpha C\).
\end{proof}

Corollary~\ref{cor:DDPComplexityRelativeGap} shows that for problems that have strictly positive cost in each stage, the proposed complexity bounds for an \(\alpha\)-relative optimal solution grow at most \emph{quadratically} for CDDP and \emph{linearly} for NDDP with respect to the number of stages \(T\). 
At the same time, they both grow \emph{exponentially} with the dimension \(d\).
In what follows, we show that this exponential dependence on \(d\) is nearly tight by providing a matching lower bound.

\subsection{Complexity Lower Bound}
\label{sec:ComplexityLB}

Note that if we take \(\delta_t=\epsilon(T-t)/(T-1)\) for \(t\in\calT\), then the complexity upper bounds in Theorems~\ref{thm:CDDPComplexityBound} and~\ref{thm:NDDPComplexityBound} depend on terms \((T-1)^{d_t+1}\) or \((T-1)^{d_t+2}\), where \(d_t\) is the state space dimension of stage \(t<T\).
It is natural to ask whether it is possible for either algorithm to achieve an \(\epsilon\)-optimal solution without the exponential dependence on the state space dimensions as in Corollary~\ref{cor:DDPComplexityRelativeGap}.
We construct a class of convex problems to show that this is indeed \emph{impossible}.

Given a \(d\)-sphere \(\calS^d(r)=\{x\in\bbR^{d+1}:\nVert{x}_2=r\}\) with radius \(r>0\), a spherical cap with depth \(\theta>0\) centered at a point \(x\in \calS^d(r)\) is the set \(\calS^d_\theta(r,x)\coloneqq\{y\in \calS^d(r):\bangle{y-x}{x}\ge -\theta r\}\).
Our construction is based on the following technical lemma, where we use \(\Gamma(\cdot)\) to denote the gamma function. 
\begin{lemma}[{\cite[Lemmas~8~and~9]{zhang2022stochastic}}]\label{lemma:ComplexitySphericalCap}
    Given a \(d\)-sphere \(\calS^d(r),d\ge 2\) and depth \(\theta<(1-\frac{\sqrt{2}}{2})r\), there exists a finite set of points \(\calW\) with 
    \[
        \abs{\calW}\ge\frac{(d^2-1)\sqrt{\pi}}{d}\frac{\Gamma(d/2+1)}{\Gamma(d/2+3/2)}\left(\frac{r}{2\theta}\right)^{(d-1)/2}, 
    \]
    such that for any \(w\in \calW\), \(\calS^d_\theta(r,w)\cap \calW=\{w\}\). 
    Moreover, given a positive constant \(l>0\), any function \(F^\calW:\calW\to(\frac{1}{2}l\theta,l\theta)\) can be extended to an \(l\)-Lipschitz continuous convex function \(F:\calB^{d+1}(r)\to\bbR\), such that \(F\) is differentiable at any point of \(\calW\), and the approximate functions 
    \begin{align*}
        \ulF_w(x)&\coloneqq\max\{F(v)+\innerprod{\nabla F(v)}{x-v}:v\in\calW,v\neq w\},\\
        \olF_w(x)&\coloneqq\conv\{F(v)+l\nVert{x-v}:v\in\calW,v\neq w\},
    \end{align*}
    satisfy \(\ulF_w(w)<0\) and \(\olF_w(w)>\frac{3}{2}l\theta\).
\end{lemma}

We next construct a class of deterministic DR-MCO problems using such convex functions, with the following parameters:
\(T\ge2\) as the number of stages,
\(L>0\) as an upper bound on Lipschitz constants,
\(d\ge 3\) as the state space dimension,
\(D=2r>0\) as the state space diameter,
and \(\epsilon>0\) as the target optimality gap.
Set \(\epsilon_t=\epsilon/(T-1)\).
Construct sets of points \(\calW_t\coloneqq\{w_{t,k}\}_{k=1}^{K_t}\), where \(\theta_t=\epsilon_t/L\) for \(t\in\calT'\).
Let \(F_t(x)=F^{\calW_t}(x)\) be an \(L\)-Lipschitz continuous convex function on \(\calB^{d}(r)\) constructed in Lemma~\ref{lemma:ComplexitySphericalCap} with \(F_t(w_{t,k})\in(\epsilon_t/2,\epsilon_t)\) for each \(w_{t,k}\in\calW_t\). 
And we denote \(F_1(x)\equiv0\).
Our DR-MCO problems for the complexity lower bound are defined by the following recursions for \(t=T,\dots,2\)
\begin{equation}\label{eq:ComplexityExample}
    \calQ_{t-1}(x_{t-1})=\min_{x_{t}\in\calB^d(r)}\left\{F_t(x_{t-1})+\calQ_{t}(x_{t})\right\},
\end{equation}
where \(\calQ_T(x)\equiv0\), and the first stage problem is defined as \(\min_{x_1\in\calB^d(r)}\calQ_1(x_1)\).
We assume that the dual bounds satisfy \(M_t\ge L\), \(t\in\calT'\) for the exactness guaranteed by Proposition~\ref{prop:LipschitzRegularizationExactness}.
To describe the asymptotic behavior of a lower bound, we use the \(\Omega\)-notation, i.e., for functions \(a,b:\bbZ\to\bbR\), \(a(s)=\Omega(b(s))\) if \(\liminf_{s\to\infty}\abs{\frac{a(s)}{b(s)}}>0\).
\begin{theorem}\label{thm:ComplexityExample}
    For the problem~\eqref{eq:ComplexityExample}, there exist single stage subproblem oracles such that the number of evaluations \(\NumEval\) for either of Algorithms~\ref{alg:ConsecutiveDualDP} and~\ref{alg:NonconsecutiveDualDP} before termination has the following lower bound
    \[
        \NumEval\ge
    \frac{d}{d-1}\sqrt{\frac{\pi}{2}(d^2-4)}\left(\frac{DL(T-1)}{2\epsilon}\right)^{d/2-1}
    =\varOmega(T^{d/2-1})\text{ as } T\to\infty.
    \]
\end{theorem}
\begin{proof}
    We assume that our SSSO at stage \(t\in\calT\) returns points in \(\calW_t\) in the first \(K_t\) evaluations without repetition.
    More specifically, for \(t\in\calT'\) and \(i_t\le K_t\), given \(\ulcQ_t^{i_t},\olcQ_t^{i_t}\) and \(x_{t-1}^{i_{t-1}}\) at some \((i_{t-1},i_t)\in\calJ_t\), the SSSO returns
    \begin{itemize}
        \item a state \(x_t^{i_t}\in\calW_t\) with \(x_t^{i_t}\neq x_t^{i'_t}\) for all \(i'_t<i_t\);
        \item a linear cut \(\calV_t^{i_t}(x):=V_t^{i_t}(x)+\ulcQ_t^{i_t}(x_t^{i_t})\);
        \item an overestimate \(v_t^{i_t}:=F_t(x_{t-1}^{i_{t-1}})+\olcQ_t^{i_t}(x_t^{i_t})\).
    \end{itemize}
    where \(V_t^{i_t}(x):=F_t(x_{t-1}^{i_{t-1}})+\bangle{\nabla F_t(x_{t-1}^{i_{t-1}})}{x-x_{t-1}^{i_{t-1}}}\) for any \((i_{t-1},i_t)\in\calJ_t\).
    Thus our under- and over-approximations of \(\calQ_t\) can be written as 
    \begin{align*}
        &\ulcQ_{t-1}^{i_{t-1}}(x)=\max\Bigl\{0,\max\bigl\{
        \ulcQ_t^{i'_t}(x_t^{i'_t})+V_t^{i'_t}(x):(i'_{t-1},i'_t)\in\calJ_t,i'_{t-1}\le i_{t-1}\bigr\}\Bigr\},\\
        &\olcQ_{t-1}^{i_{t-1}}(x)=\conv\Bigl\{\olcQ_t^{i'_t}(x_t^{i'_t})+F_t(x_{t-1}^{i'_{t-1}})+M_t\nVert{x-x_{t-1}^{i'_{t-1}}}:(i'_{t-1},i'_t)\in\calJ_t,i'_{t-1}\le i_{t-1}\Bigr\}.
    \end{align*}

    To see that the above assumption on the SSSO does not conflict with Definition~\ref{def:NoninitialStageOracle}, it is straightforward to check that the state \(x_t^{i_t}\) is feasible, the overestimate \(v_t^{i_t}=F_t(x_{t-1}^{i_{t-1}})+\olcQ_t^{i_t}(x_t^{i_t})\ge\calQ_{t-1}(x_{t-1}^{i_{t-1}})\), and the gap is controlled \(v_t^{i_t}-\calV_t^{i_t}(x_{t-1}^{i_{t-1}})=\olcQ_t^{i_t}(x_t^{i_t})-\ulcQ_t^{i_t}(x_t^{i_t})\).
    For the validness of the linear cut \(\calV_t^{i_t}\), note that by Lemma~\ref{lemma:ComplexitySphericalCap}, \(V_t^{i_t}(x)<0\) for all \(x\in\calW_{t-1}\setminus\{x_{t-1}^{i_{t-1}}\}\).
    Thus by a recursive argument from \(t=T\) to \(2\), we know that \(\ulcQ_{t-1}^{i_{t-1}}(x)\equiv 0\) when \(i_t\le K_t\) for all \(t\in\calT'\).
    Similarly, one can verify that \(\olcQ_{t-1}(x_{t-1}^{i_{t-1}})\ge \frac{3}{2}\sum_{t'\ge t}\epsilon_{t'}\).
    Therefore, \(\UB-\LB=\olcQ_1^{i_1}(x_1^{i_1})-\ulcQ_1^{i_1}(x_1^{i_1})> \frac{3}{2}\sum_{t=2}^T\epsilon_t>\epsilon\).
    In other words, when the algorithms terminate, we must have \(i_t\ge K_t\) for some \(t\in\calT'\), which implies
    \begin{align*}
    \NumEval\ge K_t&\ge \frac{((d-1)^2-1)\sqrt{\pi}}{d-1}\frac{\Gamma((d-1)/2+1)}{\Gamma((d-1)/2+3/2)}\left(\frac{rL}{\epsilon_t}\right)^{(d-2)/2}\\
    &\ge \frac{d(d-2)\sqrt{\pi}}{d-1}\frac{\Gamma(d/2+1/2)}{\Gamma(d/2+1)}\left(\frac{DL(T-1)}{2\epsilon}\right)^{(d-2)/2}\\
    &>\frac{d}{d-1}\sqrt{\frac{\pi}{2}(d^2-4)}\left(\frac{DL(T-1)}{2\epsilon}\right)^{d/2-1}.
    \end{align*}
    Here, the last inequality is due to Wendel's bound on the ratio of two gamma functions~\cite{qi2010bounds}.
\end{proof}

Compared with the existing complexity lower bound in~\cite{zhang2022stochastic}, the construction in this work is based on a deterministic DR-MCO problem.
This is only possible because our general definition of SSSO allows more flexibility of choosing the feasible state $x_t$ and valid linear cut $\calV_{t-1}$, rather than requiring them to be primal or dual optimal.

\begin{remark}
    The CDDP and NDDP algorithms (Algorithms~\ref{alg:ConsecutiveDualDP} and~\ref{alg:NonconsecutiveDualDP}) and their complexity analyses depend only on the SSSO (Definitions~\ref{def:InitialStageOracle} and~\ref{def:NoninitialStageOracle}).
    Therefore, while we only discuss SSSO implementations for finitely supported DR-MCO in Section~\ref{subsec:SSSOImplementation}, the CDDP and NDDP algorithms work for more general DR-MCO problems.
    Moreover, the complexity upper bounds (Theorems~\ref{thm:CDDPComplexityBound} and~\ref{thm:NDDPComplexityBound}, Corollary~\ref{cor:DDPComplexityRelativeGap}) and the lower bound (Theorem~\ref{thm:ComplexityExample}) remain valid for them as well.
\end{remark}

\section{Numerical Experiments}
\label{sec:Numerical}
In this section, we numerically test the proposed CDDP and NDDP algorithms.
The first test problem is a robust multi-commodity inventory problem with customer demand uncertainty.
The second test problem is a distributionally robust hydro-thermal power planning problem with stochastic energy inflows.
The computation budget consists of 10 3.0-GHz CPU cores and a total of 70 GBytes of RAM.
The algorithms are implemented using JuMP package (\cite{dunning2017jump}, v1.14) in Julia language (v1.9) with Gurobi 9.5~\cite{gurobi} as its underlying LP solver.

\subsection{Multi-Commodity Inventory Problem}
\label{subsec:Inventory}

We consider the MRCO model of multi-commodity inventory problem with uncertain customer demands and deterministic holding and backlogging costs, following the description in~\cite{georghiou_robust_2019}. 
This problem can be described through the local cost functions \(f_t\) and the uncertainty sets \(\Xi_t\).
Let \(\calJ\coloneqq\{1,2,\dots,J\}\) denote the set of product indices.
We first describe the variables in each stage \(t\in\calT\). 
We use \(x^l_{t,j}\) to denote the inventory level, \(y^a_{t,j}\) (resp. \(x^b_{t,j}\)) to denote the amount of express (resp.\ standard) order fulfilled in the current (resp.\ subsequent) stage, of some product \(j\in\calJ\).
Let \(\xi_t\in\Xi_t\) denote the uncertainty vector controlling the customer demands in stage \(t\).
The state variables consist of the collection \(x_t\coloneqq(x^l_{t,1},\dots,x^l_{t,J},x^b_{t,1},\dots,x^b_{t,J})\), while the internal variables are simply the express order amounts \(y_t=(y^a_{t,1},\dots,y^a_{t,J})\).
In stage \(t\in\calT\), the local cost function \(f_t\) and the state space \(\calX_t\) are defined by
\begin{alignat}{2}\label{eq:InventoryProblemCost}
    f_t(x_{t-1},x_t;\xi_t)\coloneqq\min_{y_t}\;&
    C^F + \sum_{j\in\calJ}\bigl(C^a_j y^a_{t,j} + C^b_j x^b_{t,j} + C^H_j[x^l_{t,j}]_+ + &&c^B_j[x^l_{t,j}]_-\bigr) \\
    \mathrm{s.t.}\; 
    & \sum_{j\in\calJ}y^a_{t,j}\le B^c,\notag\\
    & x^l_{t,j}-y^a_{t,j}-x^b_{t-1,j}=x^l_{t,j}-D_{t,j}(\xi_t), &&\forall\, j\in\calJ,\notag\\
    & y^a_{t,j}\in[0,B^a_j], &&\forall\,j\in\calJ.\notag
\end{alignat}
and
\begin{equation}\label{eq:InventoryProblemState}
    \calX_t\coloneqq\bigl\{x_t\in\bbR^{2J}:0\le x^b_{t,j}\le B^b_j,\ -B^l_j\le x^l_{t,j}\le B^l_j,\ \forall\,j\in\calJ\bigr\}.
\end{equation}
Here in the formulation, \(C^a_j\) (resp.\ \(C^b\)) denotes the express (resp.\ standard) order unit cost,
\(C^H_j\) (resp.\ \(C^B_j\)) the inventory holding (resp.\ backlogging) unit cost, 
\(B^a_j\) (resp.\ \(B^b_j\)) the bound on the express (resp.\ standard) order, 
and \(B^l_j\) the inventory level bound, for the product \(j\in\calJ\), respectively.
The first constraint in~\eqref{eq:InventoryProblemCost} is a cumulative bound \(B^c\) on the express orders, the second constraint characterizes the change in the inventory level, and the last of~\eqref{eq:InventoryProblemCost} and the state space~\eqref{eq:InventoryProblemState} represent bounds on the decision variables with respect to each product.
We also put \(C^F>0\) as a fixed cost to ensure the cost function is strictly positive (cf. Corollary~\ref{cor:DDPComplexityRelativeGap}).
We use \([x]_+\coloneqq\max\{x,0\}\) and \([x]_-\coloneqq-\min\{0,x\}\) to denote the positive and negative part of a real number \(x\).
The initial state \(x_0\) is given by \(x^l_{0,j}=x^b_{0,j}=0\) for all \(j\in\calJ\).
The uncertainty set \(\Xi_t\) is an \(E\)-dimensional box \([-1,1]^{E}\), and the customer demand is predicted by the following factor model:
\begin{equation}
    D_{t,j}(\xi_t)=\begin{cases}
        &2+\sin\left(\dfrac{(t-1)\pi}{5}\right)+\Phi_{t,j}^\transpose\xi_t,\quad j\le J/2,\\
        &2+\cos\left(\dfrac{(t-1)\pi}{5}\right)+\Phi_{t,k}^\transpose\xi_t,\quad j> J/2,
    \end{cases}
\end{equation}
where \(\Phi_{t,j}\) is an \(E\)-dimensional random vector with its entries chosen uniformly from \([-1/E,1/E]\).
Thus the value \(\Phi_{t,j}^\transpose\xi_t\in[-1,1]\) and \(D_{t,j}(\xi_t)\ge 0\) for all \(t\in\calT\) and \(j\in\calJ\).

For the following numerical test, we set the number of products \(J=5\), the number of uncertainty factors \(E=4\), \(B_j^a=B_j^b=B_j^l=10\), \(C_j^b=1\) for all \(j\in\calJ\), \(C^F=1\), and \(B^c=0.3J=1.5\).
The costs are generated uniformly at random within \(C_j^a\in[1,3]\), \(C_j^H,c_j^B\in[0,2]\) for all \(j\in\calJ\).
Due to lack of RCR of the problem~\eqref{eq:InventoryProblemCost}, we first test both CDDP (Algorithm~\ref{alg:ConsecutiveDualDP}) and NDDP (Algorithm~\ref{alg:NonconsecutiveDualDP}) with dual bounds \(M_t=1.0\times 10^2\) for all \(t\in\calT\).
The optimality gap set to be relative \(\alpha=1\%\) and approximation gaps set dynamically by \(\delta_t^{i_t}=\LB\cdot\alpha(T-t)/(T-1)\) for \(t\in\calT\).
Over 3 independently generated test cases, we plot the growth of computational time and number of oracle evaluations with respect to the number of stages in Fig.~\ref{fig:Inventory}, which shows that NDDP performs increasingly better than CDDP as the number of stages increases. 
In particular, NDDP can save up to 47.8\% computation time and up to 44.2\% subproblem oracle evaluations than CDDP on these 40-stage problems. 

\begin{figure}[htbp]
    \centering
    \begin{subfigure}{0.4\textwidth}
        \centering
        \includegraphics[width=\textwidth]{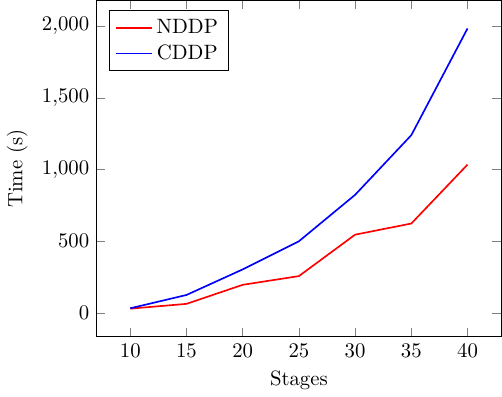}
    \end{subfigure}
    \hspace{5mm}
    \begin{subfigure}{0.4\textwidth}
        \centering
        \includegraphics[width=0.91\textwidth]{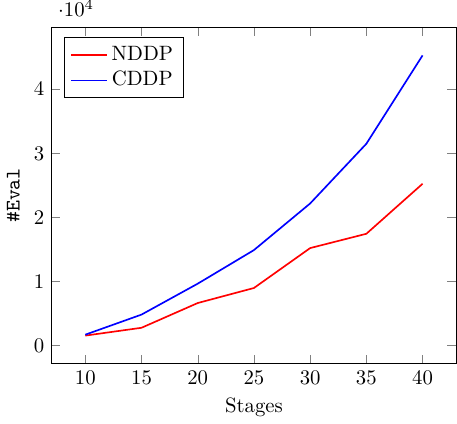}
    \end{subfigure}
    \\
    \begin{subfigure}{0.4\textwidth}
        \centering
        \includegraphics[width=\textwidth]{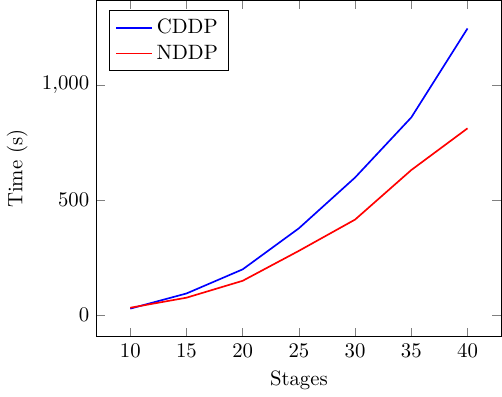}
    \end{subfigure}
    \hspace{5mm}
    \begin{subfigure}{0.4\textwidth}
        \centering
        \includegraphics[width=0.91\textwidth]{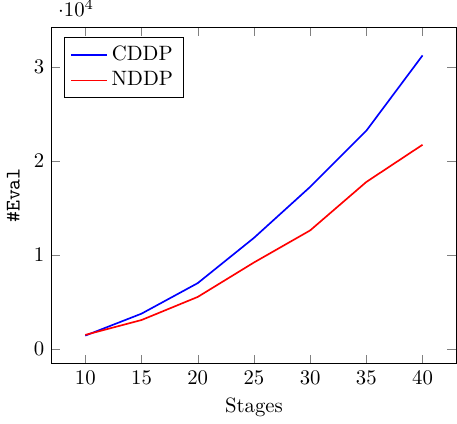}
    \end{subfigure}
    \\
    \begin{subfigure}{0.4\textwidth}
        \centering
        \includegraphics[width=\textwidth]{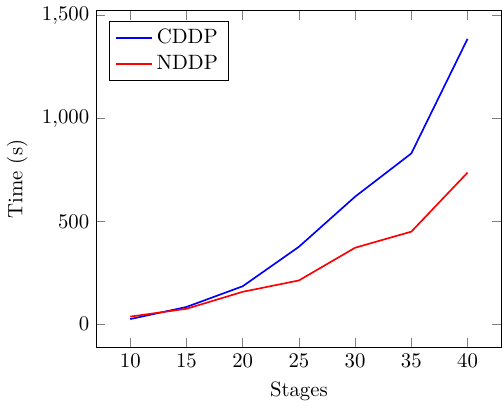}
    \end{subfigure}
    \hspace{5mm}
    \begin{subfigure}{0.4\textwidth}
        \centering
        \includegraphics[width=0.91\textwidth]{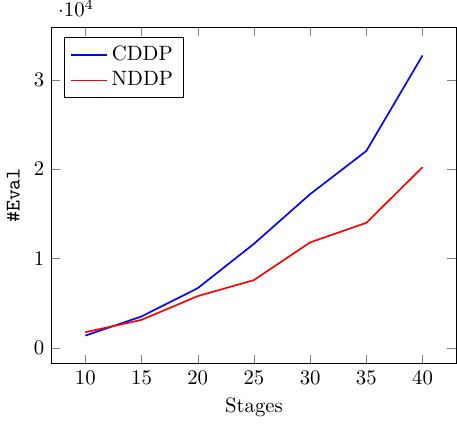}
    \end{subfigure}
    \caption{Growth of computation expenses for CDDP and NDDP}
    \label{fig:Inventory}
\end{figure}

As a comparison, we implement the NDDP algorithm without dual bounds.
Instead, it generates linear feasibility cuts for approximation of the feasible sets (see definition of feasibility cuts in, e.g.,~\cite{grothey1999note}).
We limit the computational budget to \(2000\cdot T\) subproblem oracle calls.
For 5 independently generated test cases, we have obtained the following results (Table~\ref{tab:Inventory}).
\begin{table}[htbp]
    \centering
    \caption{Comparison of NDDP Performances with or without dual bounds\label{tab:Inventory}}
    {\small
    \begin{tabular}{crrrrrrrr}
    \hline
    & \multicolumn{4}{c}{With Dual Bounds} & \multicolumn{4}{c}{Without Dual Bounds} \\
    Stage & LB & UB & Time (s) & \(\NumEval\) & LB & UB   & Time (s) & \(\NumEval\) \\
    \hline
        \multirow{5}[0]{*}{10}
        & 154.70 & 156.02 & 30.61 & 1554        & 154.70 & 155.30 & 33.34 & 1497 \\
        & 155.98 & 157.06 & 30.80 & 1523        & 155.98 & 157.37 & 14.16 & 1371 \\
        & 128.33 & 128.89 & 37.33 & 1771        & 128.31 & 129.48 & 19.46 & 1803 \\
        & 137.06 & 137.16 & 26.69 & 1355        & 137.06 & 137.06 & 14.75 & 1397 \\
        & 120.12 & 120.82 & 70.06 & 2345        & 118.11 & inf & 318.05 & 20003 \\
        \hline
        \multirow{5}[0]{*}{15}
        & 232.46 & 233.85 & 63.84 & 2776        & 232.13 & inf & 440.07 & 30037 \\
        & 233.38 & 235.56 & 74.51 & 3104        & 233.58 & 235.23 & 33.90 & 2891 \\
        & 202.31 & 203.40 & 75.25 & 3154        & 202.48 & 203.70 & 40.89 & 3522 \\
        & 208.56 & 210.14 & 61.55 & 2707        & 208.50 & 209.28 & 33.85 & 2900 \\
        & 195.16 & 196.78 & 129.38 & 4589        & 191.30 & inf & 666.52 & 30024 \\
        \hline
        \multirow{5}[0]{*}{20}
        & 291.78 & 293.94 & 196.57 & 6637        & 286.88 & inf & 615.68 & 40054 \\
        & 292.87 & 295.77 & 148.18 & 5561        & 292.87 & 293.14 & 67.45 & 4955 \\
        & 256.15 & 258.03 & 157.78 & 5810        & 256.19 & 258.63 & 74.12 & 5762 \\
        & 261.81 & 263.80 & 113.54 & 4569        & 261.81 & 261.82 & 61.22 & 4821 \\
        & 249.85 & 252.26 & 277.24 & 8296        & 235.01 & inf & 665.76 & 40002 \\
        \hline
        \multirow{5}[0]{*}{25}
        & 369.62 & 372.54 & 257.33 & 8984        & 369.66 & 371.10 & 102.86 & 7334 \\
        & 370.08 & 373.74 & 278.66 & 9225        & 370.47 & 371.39 & 101.37 & 7153 \\
        & 330.18 & 333.38 & 213.01 & 7602        & 330.37 & 331.63 & 112.88 & 8133 \\
        & 333.20 & 336.05 & 166.87 & 6346        & 333.34 & 333.74 & 96.10 & 6946 \\
        & 324.83 & 327.99 & 479.58 & 12785        & 323.43 & inf & 1167.04 & 50006 \\
        \hline
        \multirow{5}[0]{*}{30}
        & 428.81 & 432.41 & 545.98 & 15212        & 429.11 & 429.69 & 165.55 & 10217 \\
        & 429.34 & 432.08 & 414.91 & 12639        & 429.77 & 430.27 & 141.67 & 9549 \\
        & 383.74 & 387.01 & 372.00 & 11820        & 384.22 & 387.43 & 173.06 & 11368 \\
        & 386.54 & 390.41 & 277.57 & 9430        & 386.56 & 386.58 & 151.97 & 9824 \\
        & 379.46 & 382.84 & 649.69 & 16730        & 379.27 & inf & 1230.98 & 60016 \\
        \hline
    \end{tabular}}
\end{table}

In Table~\ref{tab:Inventory}, the inf indicates values of infinity or numerically infinity values (i.e., values greater than \(10^8\)) upon termination.
As we see from the table, the NDDP algorithm together with feasibility cuts fails to solve 2 out of 5 cases even when there are only 15 stages, showing the instability of the performance of feasibility cuts.
In contrast, the algorithm with the dual-bounding technique solves all of the cases within a reasonable computation time and number of subproblem oracle evaluations, without any optimality gap on those cases that both formulations are able to solve.
This demonstrates the ability of the our proposed DDP algorithms handling problems without RCR.
It is worth mentioning that for cases where the NDDP algorithm converges without dual bounds, the computation time used is usually smaller than the regularized problem, which can be explained by better numerical conditions of feasibility cuts and their effect on reducing the effective volumes of the state space.

\subsection{Hydro-Thermal Power Planning Problem}
\label{subsec:PowerSystem}

We consider the Brazilian interconnected power system described in~\cite{ding2019python}.
By assuming the stagewise independence and taking sample average approximation in the underlying stochastic energy inflow, we formulate the problem below as a finitely supported DR-MCO~\eqref{eq:DR-MSCP-Def}.
Let \(\calJ=\{1,\dots,J\}\) denote the indices of regions in the system (\(J=4\) in our data), and \(\calL=\cup_{j\in\calJ}\calL_j\) the indices of thermal power plants, where each of the disjoint subsets \(\calL_j\) is associated with the region \(j\in\calJ\).
We first describe the decision variables in each stage \(t\in\calT\). 
Let \(\Xi_t=\{\hat{\xi}_{t,1},\dots,\hat{\xi}_{t,n_t}\}\) denote the finite uncertainty set of energy inflow in stage \(t\in\calT'\).
We use \(x^l_{t,j}\) to denote the stored energy level, \(y^h_{t,j}\) to denote the hydro power generation, and \(y^s_{t,j}\) to denote the energy spillage, of some region \(j\in\calJ\); 
and \(y^g_{t,l}\) to denote the thermal power generation for some thermal power plant \(l\in\calL\).
For two different regions \(j\neq j'\in\calJ\), we use \(y^e_{t,j,j'}\) to denote the energy exchange from region \(j\) to region \(j'\), and \(y^a_{t,j,j'}\) to denote the deficit account for region \(j\) in region \(j'\).
Let \(x_t\coloneqq(x^l_{t,1},\dots,x^l_{t,J})\) be the state variables and \(y_t\) denote the internal variables consisting of \(y^h_{t,j},y^s_{t,j},y^g_{t,l},y^e_{t,j,j'}\), and \(y^a_{t,j,j'}\) for any \(j,j'\in\calJ\) and \(l\in\calL\).
Then in stage \(t\in\calT\), the state space is simply \(\calX_t\coloneqq\prod_{j=1}^J[0,B_j^l]\), and the local cost function \(f_t\) can be defined as
\begin{alignat}{2}\label{eq:PowerSystem}
    f_t(x_{t-1}&,x_t;\xi_t)\coloneqq\\
    \min_{y_t}\quad&\sum_{j\in\calJ}\bigg( C^s y^s_{t,j} + \sum_{l\in\calL_j}C^g_l y^g_{t,l} + \sum_{j'\in\calJ}\big(C^e_{j,j'} y^e_{t,j,j'} + C^a_{j,j'} y^a_{t,j,j'}\big) \bigg) &&\notag\\
    \mathrm{s.t.}\quad& x^l_{t,j}+y^h_{t,j}+y^s_{t,j}=x^l_{t-1,j}+\xi_{t,j}, &&\forall\,j\in\calJ,\notag\\
    & y^h_{t,j}+\sum_{l\in\calL_j}y^g_{t,l}+\sum_{j'\in\calJ}(y^a_{t,j,j'}-y^e_{t,j,j'}+y^e_{t,j',j})=D_{t,j}, &&\forall\,j\in\calJ,\notag\\
    & y^h_{t,j}\in[0,B^h_j], &&\forall\,j\in\calJ,\notag\\
    & y^g_{t,l}\in[B^{g,-}_l,B^{g,+}_l], &&\forall\,l\in\calL,\notag\\
    & y^a_{t,j,j'}\in[0,B^a_{j,j'}],\ y^e_{t,j,j'}\in[0,B^e_{j,j'}], &&\forall\,j,j'\in\calJ.\notag
\end{alignat}
Here in the formulation, \(C^s\) denotes the unit penalty on energy spillage,
\(C^g_l\) the unit cost of thermal power generation of plant \(l\),
\(C^e_{j,j'}\) the unit cost of power exchange from region \(j\) to region \(j'\),
\(C^a_{j,j'}\) the unit cost on the energy deficit account for region \(j\) in region \(j'\),
\(D_{t,j}\) the deterministic power demand in stage \(t\) and region \(j\),
\(B^l_j\) the bound on the storage level in region \(j\),
\(B^h_j\) the bound on hydro power generation in region \(j\),
\(B^{g,-}_l,B^{g,+}_l\) the lower and upper bounds of thermal power generation in plant \(l\),
\(B^a_{j,j'}\) the bound on the deficit account for region \(j\) in region \(j'\),
and \(B^e_{j,j'}\) the bound on the energy exchange from region \(j\) to region \(j'\).
The first constraint in~\eqref{eq:PowerSystem} characterizes the change of energy storage levels in each region \(j\), the second constraint imposes the power generation-demand balance for each region \(j\), and the rest are bounds on the decision variables.
The initial state \(x_0\) and \(\xi_1\) are given by data.

The energy inflow outcomes are sampled from multivariate lognormal distributions that are interstage independent.
Then the distributional ambiguity set is constructed using Wasserstein metric to reduce the effect of overtraining with the sampled outcome, according to~\cite{duque2019distributionally}.
To be precise, suppose \(\hat{p}_{t}=(1/n_t,\dots,1/n_t)\in\Delta^{n_t}\) is an empirical distribution of outcomes \(\Xi_t=\{\hat{\xi}_{t,1},\dots,\hat{\xi}_{t,n_t}\}\).
Then, the distributional ambiguity set \(\calP_{t}\) is described by
\begin{equation}\label{eq:HydroThermalWassersteinSet}
    \calP_{t}\coloneqq\left\{p_{t}\in\Delta^{n_t}:W_t(p_{t},\hat{p}_{t})\le\rho_t\right\},
\end{equation}
for some radius \(\rho_t\ge0\), where the Wasserstein metric \(W_t\) for finitely supported distributions is defined by
\begin{alignat}{2}
    W_t(p_{t},\hat{p}_{t})\coloneqq\min_{\pi_{k,k'}\ge0}\quad&\sum_{k,k'=1}^{n_t}\nVert{\hat{\xi}_{t,k}-\hat{\xi}_{t,k'}}\pi_{k,k'}\label{eq:DefWassersteinDist}\\
    \mathrm{s.t.}\quad&\sum_{k=1}^{n_t}\pi_{k,k'}=p_{t,k}, &&\forall\,k'=1,\dots,n_t,\notag\\
    &\sum_{k'=1}^{n_t}\pi_{k,k'}=\hat{p}_{t,k}=\frac{1}{n_t},\quad &&\forall\,k=1,\dots,n_t.\notag
\end{alignat}
Note when the radius \(\rho_t=0\), the ambiguity set \(\calP_{t}\) becomes a singleton.
In our numerical tests, we choose the radius to be relative to the total distances, i.e., \(\rho_t=\beta\cdot\sum_{k,k'=1}^{n_t}\nVert{\hat{\xi}_{t,k}-\hat{\xi}_{t,k'}}\) for some \(\beta\ge0\).
At the same time, we use uniform dual bounds for the tests, i.e., \(M_t=M>0\) for all \(t\in\calT'\).
When the relative optimality gap \(\alpha\) is smaller than the threshold \(5\%\), we check whether all the active cuts in the recent iterations are strictly smaller than the dual bound.
If they are, then the algorithm is terminated, and otherwise the dual bound \(M\) is increased by a factor of \(\sqrt{10}\approx 3.1623\) with all the over-approximations reset to \(\ulcQ_t^i(x)\leftarrow+\infty\), \(t\in\calT\).
Five scenarios are sampled independently in each stage for the nominal problem \(n_t=5\) before the distributional robust counterpart is constructed by \eqref{eq:HydroThermalWassersteinSet}.
For the 24-stage problem that we consider, the samples already give a total \(5^{24}\approx 5.9\times 10^{16}\) scenario paths, which is already impossible to solve via an extensive formulation.
We have then obtained the following results (Table~\ref{tab:HydroThermal}).
\begin{table}[htbp]
    \centering
    \caption{CDDP on hydro-thermal power planning with different radii\label{tab:HydroThermal}}
    {\small
    \begin{tabular}{crrrrr}
    \hline
    $\beta$ & $\log_{10}(M)$ & LB ($\cdot10^7$) & UB ($\cdot10^7$) & Med.\ Time (s) & Med.\ $\NumEval$ \\
    \hline
        \multirow{5}[0]{*}{0.00}
        & 2 & 4.60 & 4.84 & 784.45 & 33075 \\
        & 3 & 4.59 & 4.83 & 301.06 & 15345 \\
        & 4 & 4.58 & 4.82 & 374.88 & 17505 \\
        & 5 & 4.58 & 4.82 & 378.80 & 17595 \\
        & 6 & 4.58 & 4.82 & 386.86 & 17910 \\
        \hline
        \multirow{5}[0]{*}{0.02}
        & 2 & 4.84 & 5.09 & 439.28 & 22095 \\
        & 3 & 4.84 & 5.09 & 263.26 & 13725 \\
        & 4 & 4.82 & 5.07 & 365.72 & 17100 \\
        & 5 & 4.82 & 5.07 & 358.08 & 16830 \\
        & 6 & 4.82 & 5.07 & 405.42 & 18090 \\
        \hline
        \multirow{5}[0]{*}{0.04}
        & 2 & 5.10 & 5.36 & 921.99 & 36720 \\
        & 3 & 5.06 & 5.32 & 237.34 & 13095 \\
        & 4 & 5.06 & 5.33 & 327.38 & 15975 \\
        & 5 & 5.06 & 5.33 & 325.62 & 15930 \\
        & 6 & 5.06 & 5.33 & 350.65 & 16650 \\
        \hline
        \multirow{5}[0]{*}{0.06}
        & 2 & 5.34 & 5.62 & 606.61 & 27810 \\
        & 3 & 5.32 & 5.60 & 284.40 & 14445 \\
        & 4 & 5.31 & 5.59 & 277.21 & 14310 \\
        & 5 & 5.30 & 5.58 & 271.26 & 14130 \\
        & 6 & 5.31 & 5.59 & 285.30 & 14445 \\
        \hline
        \multirow{5}[0]{*}{0.08}
        & 2 & 5.57 & 5.86 & 519.87 & 25380 \\
        & 3 & 5.55 & 5.84 & 204.36 & 11790 \\
        & 4 & 5.54 & 5.83 & 290.96 & 14715 \\
        & 5 & 5.54 & 5.83 & 223.92 & 12510 \\
        & 6 & 5.54 & 5.83 & 257.26 & 13500 \\
        \hline
        \multirow{5}[0]{*}{0.10}
        & 2 & 5.81 & 6.12 & 446.94 & 22140 \\
        & 3 & 5.78 & 6.08 & 194.85 & 10980 \\
        & 4 & 5.78 & 6.09 & 262.42 & 12915 \\
        & 5 & 5.78 & 6.09 & 215.10 & 11790 \\
        & 6 & 5.79 & 6.09 & 237.22 & 12420 \\
        \hline
    \end{tabular}}
\end{table}
  
  In Table~\ref{tab:HydroThermal}, the lower bound (LB), the upper bound (UB) at termination, the computation time (Med.\,Time) and the number of subproblem oracles (Med.\,\(\NumEval\)) shown are the median of the five test cases.
  The logarithmic dual bound \(\log_{10}(M)\) listed in the table correspond to the initial dual bound.
  We see that for different choices of the relative radii \(\beta\), the median computation time and number of subproblem oracle evaluations are usually smaller when \(\log_{10}(M)=3\), without compromising the quality of upper and lower bounds.
  This can be explained by the better numerical conditions for the smaller dual bound (the cuts have smaller Lipschitz constants), leading to shorter subproblem oracle evaluations times (cf. Algorithm~\ref{alg:SubproblemOracle}).
  We thus conclude that the dual-bounding technique could lead to smaller number of subproblem oracle evaluations, as well as shorter computation time for a given DR-MCO problem.

\section{Concluding Remarks}
\label{sec:Conclusion}

In this work, we extend the CDDP algorithm to DR-MCO problems and more importantly, we propose a new NDDP algorithm that achieves an optimal single stage subproblem oracle complexity when terminated with relative optimality criteria.
Subproblem oracles are defined as an abstraction of the usual subproblem solution subroutines, with an illustration on finite uncertainty sets in Algorithm~\ref{alg:SubproblemOracle}.
An interesting extension is to develop implementations of subproblem oracles on possibly infinite uncertainty sets.
Moreover, on the practical side, an important question is how the out-of-sample performance compares with the usual MSCO and MRCO models.
Finally, while our CDDP and NDDP algorithms can converge with deterministic upper bounds, the sampling-based stochastic DDP algorithms usually leads to good lower bounds within fewer iterations on MSCO problems.
It is then of interest to see any warm-start methods of our CDDP or NDDP algorithm through sampling.
These directions will be investigated in our subsequent work~\cite{zhang2022distributionally}.

\bibliographystyle{plainnat}
\bibliography{ref_rob_stoch,ref_misc}

\clearpage
\appendix
\section{Instantiations of DR-MCO Framework}
\label{sec:Instantiations}

In this section, we instantiate our DR-MCO framework for some commonly studied problems in the literature, e.g.,~\cite{duque2019distributionally}, \cite{georghiou_robust_2019},  and~\cite{philpott_distributionally_2018}.

\subsection{More Examples of Ambiguity Sets}

Besides the MSCO and the polyhedral MRCO that are discussed in Section~\ref{sec:FiniteUncertaintySet}, we now present some more examples of ambiguity sets constructed using $\phi$-divergences or Wasserstein distances.

We begin with the ambiguity sets defined by $\phi$-divergences.
For simplicity, we restrict our attention to the case where there is a finitely supported nominal probability measure $\hat{p}_t$ in each stage $t$, which is a standard assumption in many single- or two-stage problems~\cite{ben2013robust,bayraksan2015data}.
Without loss of generality, we may assume that the finite support set of $\hat{p}_t$ is contained in our finite uncertainty set
$\Xi_t=\{\hat{\xi}_{t,1},\dots,\hat{\xi}_{t,n_t}\}$.  
Let $\phi:(0,+\infty)\to\bbR\cup\{+\infty\}$ be a proper, lower semicontinuous convex function such that 
$\phi(1)=0$.
Then the $\phi$-divergence ambiguity set of size $\rho_t>0$ is defined as
\begin{equation}\label{eq:app:PhiDivergence}
    \calP_t:=\left\{p_t=(p_{t,1},\dots,p_{t,n_t})\in\Delta^{n_t}:\sum_{j=1}^{n_t}\hat{p}_{t,j}\phi\Big(\frac{p_{t,j}}{\hat{p}_{t,j}}\Big)\le\rho_t\right\}.
\end{equation}
Here, we allow the possibility of the denominator $\hat{p}_{t,i}$ being zero by taking the convention that $0\phi(a/0):=a\bar{\phi}$ for any $a>0$ and $0\phi(0/0):=0$, where $\bar{\phi}:=\lim_{b\to+\infty}\phi(b)/b$ is a parameter determined by our choice of $\phi$, possibly taking the value $+\infty$.
We will also use the notation of (Fenchel) conjugate function $\phi^*(a):=\sup_{b\ge0}\{ab-\phi(b)\}$ of $\phi$.
Popular choices for $\phi$ include
\begin{itemize}
    \item Kullback-Leibler divergence: $\phi(b)=b\log(b)-b+1$, with $\phi^*(a)=\exp(a)-1$; 
    \item variation distance: $\phi(b)=|b-1|$, with
    \[
        \phi^*(a)=\begin{cases}
            -1,\quad&a<-1,\\
            a,\quad&-1\le a\le1;
        \end{cases}
    \]
    and
    \item modified $\chi^2$-divergence: $\phi(b)=(b-1)^2$, with 
    \[
        \phi^*(a)=\begin{cases}
            -1,\quad &a<-2,\\
            a+a^2/4,\quad &a\ge-2.
        \end{cases}
    \]
    The modified $\chi^2$-divergence has been used in the DR-MLO model in~\cite{philpott_distributionally_2018}.
\end{itemize}
In all these cases, lines 7-8 of our SSSO implementation described in Algorithm~\ref{alg:SubproblemOracle} can be efficiently solved through reformulation into linear conic optimization problems (and in particular, linear optimization problems for the variation distance, and second-order conic optimization problems for the modified $\chi^2$-divergence).
Alternatively, following the duality argument for $\phi$-divergences in~\cite{ben2013robust,bayraksan2015data}, we can also consider a reformulation of the recursion problem~\eqref{eq:DR-MSCP-Recursion} as 
\begin{equation}\label{eq:app:ReformulationPhiDivergence}
    \begin{aligned}
        \calQ_{t-1}(x_{t-1})=\min_{\mu,\nu,\alpha_{t,k},x_{t,k}}\quad &\mu+\rho_t\nu+\nu\sum_{k=1}^{n_t}\hat{p}_{t,k}\phi^*\Big(\frac{\alpha_{t,k}-\mu}{\nu}\Big)\\
        \mathrm{s.t.}\quad &\alpha_{t,k}\ge f_t(x_{t-1},x_{t,k};\hat{\xi}_{t,k})+\calQ_t(x_{t,k}),\ &&k=1,\dots,n_t,\\
        & \alpha_{t,k}-\mu\le\bar{\phi}\cdot\nu,\quad &&k=1,\dots,n_t,\\
        & x_{t,k}\in\calX_t,\quad &&k=1,\dots,n_t,\\
        &\nu\ge0,
    \end{aligned}
\end{equation}
where the constraint $\alpha_{t,k}-\mu\le\bar{\phi}\cdot\nu$ is discarded when $\bar{\phi}=+\infty$.
This reformulation~\eqref{eq:app:ReformulationPhiDivergence} allows us to use only one duplicating variable $z_t$ in a bounded Lagrangian dual problem to generate a linear cut:
\begin{equation}\label{eq:app:BoundedDualPhiDivergence}
    \begin{aligned}
        \max_{\nVert{\lambda}_*\le M_t}\min_{\substack{\mu,\nu,z_t\\\alpha_{t,k},x_{t,k}}}\quad &\mu+\rho_t\nu+\nu\sum_{k=1}^{n_t}\hat{p}_{t,k}\phi^*\Big(\frac{\alpha_{t,k}-\mu}{\nu}\Big)+\bangle{\lambda}{x_{t-1}-z_t}\\
        \mathrm{s.t.}\quad &\alpha_{t,k}\ge f_t(z_t,x_{t,k};\hat{\xi}_{t,k})+\calQ_t(x_{t,k}),\ &&k=1,\dots,n_t,\\
        & \alpha_{t,k}-\mu\le\bar{\phi}\cdot\nu,\quad &&k=1,\dots,n_t,\\
        & x_{t,k}\in\calX_t,\quad &&k=1,\dots,n_t,\\
        &\nu\ge0.
    \end{aligned}
\end{equation}

Another family of ambiguity sets are defined by Wasserstein distances.
Here, similar to the $\phi$-divergence case above, we briefly derive the reformulation of bounded Lagrangian dual problems when the support set is finite (see definition in~\eqref{eq:HydroThermalWassersteinSet} and \eqref{eq:DefWassersteinDist}).
DR-MCO with Wasserstein ambiguity sets involving infinite support sets will be discussed in our subsequent work~\cite{zhang2022distributionally}.

Let $\hat{p}_t=(\hat{p}_{t,1},\dots,\hat{p}_{t,n_t})\in\Delta^{n_t}$ denote the center of of a Wasserstein ball of radius $\rho_t>0$.
As discussed in~\cite{duque2019distributionally}, the recursion problem~\eqref{eq:DR-MSCP-Recursion} with Wasserstein ambiguity sets~\eqref{eq:HydroThermalWassersteinSet} can be reformulated as
\begin{equation}\label{eq:app:Wasserstein}
    \begin{aligned}
        \calQ_{t-1}(x_{t-1})=
        \min_{\beta,\alpha_{t,k},x_{t,k}}\quad &\rho_t\beta+\sum_{k=1}^{n_t}\hat{p}_{t,k}\alpha_{t,k}\\
        \mathrm{s.t.}\quad &\nVert{\hat{\xi}_{t,k}-\hat{\xi}_{t,k'}}\beta+\alpha_{t,k} && \\
        &\quad \ge f_t(x_{t-1},x_{t,k};\hat{\xi}_{t,k'})+\calQ_t(x_{t,k}),\quad &&k,k'=1,\dots,n_t,\\
        &x_{t,k}\in\calX_t,\quad &&k=1,\dots,n_t,\\
        &\beta\ge0.
    \end{aligned}
\end{equation}
Then in our context, the bounded Lagrangian dual problem can be written as
\begin{equation}\label{eq:app:BoundedDualWasserstein}
    \begin{aligned}
        \max_{\nVert{\lambda_t}_*\le M_t}
        \min_{\substack{\beta,\alpha_{t,k},\\z_t,x_{t,k}}}\quad &\rho_t\beta+\sum_{k=1}^{n_t}\hat{p}_{t,k}\alpha_{t,k}+\bangle{\lambda_t}{x_{t-1}-z_t}\\
        \mathrm{s.t.}\quad &\nVert{\hat{\xi}_{t,k}-\hat{\xi}_{t,k'}}\beta+\alpha_{t,k} && \\
        &\quad \ge f_t(z_t,x_{t,k};\hat{\xi}_{t,k'})+\calQ_t(x_{t,k}),\quad &&k,k'=1,\dots,n_t,\\
        &x_{t,k}\in\calX_t,\quad &&k=1,\dots,n_t,\\
        &\beta\ge0.
    \end{aligned}
\end{equation}

An alternative implementation of SSSO can be based on~\eqref{eq:app:BoundedDualPhiDivergence} and~\eqref{eq:app:BoundedDualWasserstein}, which we summarize in Algorithm~\ref{alg:app:SubproblemOracleCentralized}.
We can see that the alternative SSSO implementation is simpler in the sense that it does not involve any for-loop.
While in general, decomposing into subproblems for each outcome $k=1,\dots,n_t$ is more favorable for solution efficiency and parallelization, on small instances Algorithm~\ref{alg:app:SubproblemOracleCentralized} could avoid the overhead for the for-loop.
Moreover, the instantiations~\eqref{eq:app:PhiDivergence} and~\eqref{eq:app:Wasserstein} together with Algorithm~\ref{alg:app:SubproblemOracleCentralized} show that our complexity analyses easily apply to the existing works on DR-MCO with $\phi$-divergence and Wasserstein distance ambiguity sets~\cite{philpott_distributionally_2018,duque2019distributionally}.

\begin{algorithm}[ht]
    \caption{An Alternative Realization of Noninitial Stage Subproblem Oracle}
    \label{alg:app:SubproblemOracleCentralized}
    \begin{algorithmic}[1]
        \Require{\(\Xi_t=\{\hat{\xi}_{t,1},\dots,\hat{\xi}_{t,n_t}\}\) and \((x_{t-1},\ulcQ_t,\olcQ_t)\)} 
        \Ensure{\((\calV_{t-1},v_{t-1},x_t;\gamma_t)\) satisfying Definition~\ref{def:NoninitialStageOracle}}
        \State{solve the bounded Lagrangian dual problem~\eqref{eq:app:BoundedDualPhiDivergence} or \eqref{eq:app:BoundedDualWasserstein}} 
        \State{collect the solutions \((\hat{x}_{t,k},\hat{z}_{t};\hat{\lambda}_{t})\) and the optimal value \(\ubar{v}_{t}\)}
        \State{define \(\calV_t(x)\coloneqq\ubar{v}_{t}+\bangle{\hat{\lambda}_{t}}{x-x_{t-1}}\)}
        \State{calculate \(\gamma_{t,k}\coloneqq\olcQ_t(\hat{x}_{t,k})-\ulcQ_t(\hat{x}_{t,k})\)}
        \State{let \(\bar{v}_{t,k}:=f_t(x_{t-1},\hat{x}_{t,k};\hat{\xi}_{t,k})+\olcQ_t(\hat{x}_{t,k})+M_t\nVert{x_{t-1}-\hat{z}_{t,k}}\)}
        \State{calculate \(v_{t-1}\coloneqq\max_{p_{t}\in\calP_{t}}\sum_{k=1}^{n_t}p_{t,k}\bar{v}_{t,k}\)}
        \State{find \(k^*\) such that \(\gamma_{t,k^*}\ge\gamma_{t,k}\) for all \(k=1,\dots,n_t\) and set \(x_t\coloneqq \hat{x}_{t,k^*},\gamma_t\coloneqq\gamma_{t,k^*}\)}
    \end{algorithmic}
\end{algorithm}

\subsection{Distributionally Robust Multistage Linear Optimization}
\label{subsec:DR-MLO}

We now discuss a special subclass, distributionally robust multistage linear optimization (DR-MLO), of the more general DR-MCO problems.
The models used in our numerical experiments in Section~\ref{sec:Numerical} fall within this DR-MLO subclass.
To be most concrete, we switch to the standard vector dot product notation (e.g., $c^\transpose y$) from the previously used inner product notation $\bangle{\cdot}{\cdot}$.
Using auxiliary variables $y_t$ as in Example~\ref{ex:ConstrainedOptimization}, let the local cost function $f_t$ in stage $t$ be defined by 
\begin{equation}\label{eq:app:DR-MLO}
    \begin{aligned}
        f_t(x_{t-1},x_t;\xi_t):=\min_{y_t}\quad
        & c_t^\transpose y_t\\
        \mathrm{s.t.}\quad & \calA_t(\xi_t)x_{t-1}+B_tx_t+E_ty_t\ge F_t\xi_t+g_t,
    \end{aligned}
\end{equation}
for some matrix-valued linear map $\calA_t(\xi_t)$, matrices $B_t,E_t,F_t$ and vectors $c_t,g_t$ of appropriate dimensions.
While the state variable $x_t$ does not directly appear in the objective function, we can use the standard trick of increasing the dimension of $y_t$ by 1 and setting the added component to represent the cost incurred by $x_t$.
The DR-MLO~\eqref{eq:app:DR-MLO} has a special uncertainty structure: the coefficients of $x_{t-1}$ and the right-hand side in the constraint depend affine linearly in the uncertainty vector $\xi_t$, while the coefficients of $x_t$ and $y_t$ are deterministic.
The same structure was adopted by the existing multistage robust linear optimization (MRLO) framework~\cite{georghiou_robust_2019}, and we will see below that it conforms to our assumptions on $f_t$ and facilitates algorithmic discussion.
The recursion~\eqref{eq:DR-MSCP-Recursion} can be written as $\calQ_{t-1}(x_{t-1})=\sup_{p_t\in\calP_t}\bbE_{\xi_t\sim p_t}Q_t(x_{t-1};\xi_t)$ where
\begin{equation}\label{eq:app:DR-MLO-Recursion}
    \begin{aligned}
    Q_{t}(x_{t-1};\xi_t)=\min_{x_t,y_t}\quad &c_t^\transpose y_t+\calQ_t(x_t)\\
    \mathrm{s.t.}\quad& \calA_t(\xi_t)x_{t-1}+B_tx_t+E_ty_t\ge F_t\xi_t+g_t.
    \end{aligned}
\end{equation}

Some remarks about~\eqref{eq:app:DR-MLO} and~\eqref{eq:app:DR-MLO-Recursion} are in order.
First, it is clear that for each fixed $\xi_t$, $f_t$ is a convex function in $(x_{t-1},x_t)$ since it is polyhedrally representable.
Second, let $\calX_t$ and $\calY_t$ denote the projection of the feasibility set in~\eqref{eq:app:DR-MLO} onto the $x_t$-space and $y_t$-space, respectively.
Then under the assumption that $\calY_t$ is compact, we see that $f_t$ is lower semicontinuous, which by Lemma~\ref{prop:RecursionConvexity} ensures that the expectation $\bbE_{\xi_t\sim p_t}Q_t(x_{t-1};\xi_t)$ is well-defined for any Borel probability measure $p_t$.
Third, assuming $\calX_t$ is also compact and nonempty, then the simultaneous minimization of $x_t$ and $y_t$ in~\eqref{eq:app:DR-MLO-Recursion} is also well-defined.
Fourth, the nonnegativity assumption on $f_t$ is often satisfied by many DR-MLO applications, e.g., see the multi-commodity inventory problems and the hydro-thermal power planning problems in Section~\ref{sec:Numerical}.

Next we derive explicitly the bounded Lagrangian dual~\eqref{eq:RecursionLagrangianBoundedDual} for a fixed $\xi_t=\hat{\xi}_{t}$, the solution of which is the most critical step in the implementation of SSSO in Algorithm~\ref{alg:SubproblemOracle} (lines 2-3).
Suppose the under-approximation $\ulcQ_t(x_t)$ is given by
\begin{equation}\label{eq:app:DR-MLO-UnderApprox}
    \ulcQ_t(x_t)=\max_{j=1,\dots,i}\ (\hat{\kappa}_t^{j})^\transpose x_t+\hat{\upsilon}_t^j,
\end{equation}
for some vectors $\hat{\kappa}_t^1,\dots,\hat{\kappa}_t^i$ and numbers $\hat{\upsilon}_t^1,\dots,\hat{\upsilon}_t^i$.
Moreover, suppose the norm $\nVert{x_t}$ on $\bbR^{d_t}$ is polyhedrally representable, i.e., there exist a matrix $N\in\bbR^{e_t\times d_t}$ such that $\nVert{x_t}=\max_{j=1,\dots,e_t}(Nx_t)_j$.
Then the dual norm ball of radius $M_t$ in~\eqref{eq:RecursionLagrangianBoundedDual} can be written as
\[
    \nVert{\lambda_t}_*\le M_t\iff\exists\,w\in\bbR^{e_t}_{\ge0},\;\mathrm{s.t.}\ \sum_{j=1}^{e_t}w_j\le M_t,\ \sum_{j=1}^{e_t}w_jN_j=\lambda_t,
\]
where $N_j$ stands for the $j$-th row of $N$.
Using linear optimization strong duality for the inner minimization, the dual problem~\eqref{eq:RecursionLagrangianBoundedDual} at a given state $x_{t-1}=\hat{x}_{t-1}$ can be written as
\begin{equation}\label{eq:app:DualLinearRecursion}
    \begin{aligned}
        &&\max_{\nVert{\lambda_t}_*\le M_t}\quad\min_{x_t,y_t,z_t,\theta_t}\quad & c_t^\transpose y_t+\lambda_t^\transpose (\hat{x}_{t-1}-z_t)+\theta_t\\
        &&\mathrm{s.t.}\quad & \calA_t(\hat{\xi}_t)z_t+B_tx_t+E_ty_t\ge F_t\hat{\xi}_t+g_t,\\
        && & \theta_t\ge (\hat{\kappa}_t^j)^\transpose x_t+\hat{\upsilon}_t^j,\quad j=1,\dots,i,\\
        =&&\max_{\lambda_t,w_t,\zeta_t,\eta_{t,1},\dots,\eta_{t,i}}\quad & (F_t\hat{\xi}_t+g_t)^\transpose \zeta_t+\sum_{j=1}^{i}\hat{\upsilon}_t^j\eta_{t,j}\\
        && \mathrm{s.t.}\quad & B_t^\transpose \zeta_t+\sum_{j=1}^{i}\eta_{t,j}\hat{\kappa}_t^j=0,\\
        && & E_t^\transpose\zeta_t-c_t=0,\\
        && & \calA_t(\hat{\xi}_t)^\transpose\zeta_t+\lambda_t=0,\\
        && & \sum_{j=1}^{i}\eta_{t,j}=1,
        \sum_{j=1}^{e_t}w_j\le M_t,\ \sum_{j=1}^{e_t}w_jN_j=\lambda_t,\\
        && & w_t\ge0,\zeta_t\ge0,\eta_{t,1},\dots,\eta_{t,i}\ge0.
    \end{aligned}
\end{equation}
The dual problem~\eqref{eq:app:DualLinearRecursion} is again a linear optimization problem, which allows efficient solution through linear optimization solvers as subroutines.

We remark that a potential benefit of considering DR-MLO is that it is possible to implement SSSO with $M_t=+\infty$ (i.e., the usual unbounded Lagrangian dual problems) while ensuring convergence of CDDP and NDDP in finite time.
The argument has been used for showing convergence of dual dynamic programming algorithms on multistage stochastic linear optimization problems~\cite{philpott_convergence_2008,shapiro_analysis_2011}, multistage robust linear optimization problem~\cite{georghiou_robust_2019}, and DR-MLO problems~\cite{duque2019distributionally}.
To the best of our knowledge, it remains open whether such implementation still enjoys complexity upper bounds that are polynomial in $T$.

\end{document}